\documentclass[11pt]{article}
\usepackage{latexsym,amsfonts,amsthm,amsmath,amscd,amssymb}
\usepackage[dvips]{graphicx}

\newtheorem{lemma}{Lemma}[section]
\newtheorem{thm}[lemma]{Theorem}
\newtheorem{rem}[lemma]{Remark}
\newtheorem{prop}[lemma]{Proposition}

\newcommand\matC{{\mathbb{C}}}

\renewcommand{\hbar}{{\overline{h}}}

\newfont{\Got}{eufm10 scaled 1200}

\newcommand{\mycap} [1] {\caption{\footnotesize{#1}}}

\newcommand\calF{{\mathcal F}}
\newcommand\calS{{\mathcal S}}

\begin{document}

\title{Generic flows on $3$-manifolds}

\author{Carlo~\textsc{Petronio}}

\maketitle

\begin{abstract}
We provide a combinatorial presentation of the
set $\calF$ of 3-dimensional \emph{generic flows}, namely the set
of pairs $(M,v)$ with $M$ a compact oriented $3$-manifold
and $v$ a nowhere-zero vector field on $M$ having generic behaviour along $\partial M$,
with $M$ viewed up to diffeomorphism and $v$ up to homotopy on $M$ fixed on $\partial M$.
To do so we introduce a certain class $\calS$ of finite 2-dimensional polyhedra with
extra combinatorial structures, and some moves on $\calS$, exhibiting
a surjection $\varphi:\calS\to\calF$ such that $\varphi(P_0)=\varphi(P_1)$ if
and only if $P_0$ and $P_1$ are related by the moves. To obtain this
result we first consider the subset $\calF_0$ of $\calF$ consisting of flows
having all orbits homeomorphic to closed segments or points,
constructing a combinatorial counterpart $\calS_0$ for $\calF_0$ and
then adapting it to $\calF$.
\end{abstract}

\noindent
Combinatorial presentations of $3$-dimensional topological categories, such
as the description of closed oriented $3$-manifolds via surgery along framed links in $S^3$,
and many more, have proved crucial for the theory of quantum invariants, initiated
in~\cite{RT} and~\cite{TV} and now one of the main themes of geometric topology.
In this paper we provide one such presentation for the set $\calF$ of pairs $(M,v)$ with
$M$ a $3$-manifold and $v$ a flow
having generic behaviour on $\partial M$, viewed up to homotopy fixed on $\partial M$. This extends the presentation
of closed combed $3$-manifolds contained in~\cite{LNM}, and it is based on a generalization
of the notion of \emph{branched spine}, introduced there as a combination
of the definition of special spine due to Matveev~\cite{Matveev:AAM}
with the concept of branched surface introduced by Williams~\cite{Williams}, already
partially investigated by Ishii~\cite{Ishii} and Christy~\cite{Christy}.
A \emph{presentation} here is as usual meant as a constructive surjection onto $\calF$ from a set
of finite combinatorial objects, together with a finite set of combinatorial
moves on the objects generating the equivalence relation induced by the surjection.

To get our presentation we will initially restrict to generic flows having all orbits
homeomorphic to points or to segments, viewed first up to diffeomorphism and then up to homotopy,
and we will carefully describe their combinatorial counterparts.

A restricted type of generic flows on manifolds with boundary
was actually already considered in~\cite{LNM}, but two such flows could never
be glued together along boundary components. On the contrary, as we will point out in detail in
Remark~\ref{glue:rem}, using the flows we consider here one can develop a theory
of cobordism and hence, hopefully, a TQFT in the spirit of~\cite{Turaev}.
Another reason why we expect that our
encoding of generic flows might have non-trivial applications is that
the notion of branched spine was one of the combinatorial tools
underlying the theory of quantum hyperbolic invariants of
Baseilhac and Benedetti~\cite{BB1,BB2,BB3}.

\bigskip

\textsc{Acknowledgements} The author profited from several inspiring discussions with Riccardo Benedetti.

\section{Generic flows, streams, and stream-spines}

In this section we define the topological objects that we will deal with
in the paper and we introduce the combinatorial objects
that we will use to encode them. We then describe our first representation result,
for manifolds with generic traversing flows (that we call \emph{streams}) viewed up
to diffeomorphism.

\subsection{Generic flows}
Let $M$ be a smooth, compact, and oriented $3$-manifold with non-empty boundary,
and let $v$ be a nowhere-vanishing vector field on $M$. We will always in this paper assume
the following genericity of the tangency of $v$ to $\partial M$, first discussed by Morin~\cite{Morin}:

\begin{itemize}
\item[(\textbf{G1})]
The field $v$ is tangent to $\partial M$ only along a union $\Gamma$ of circles,
and $v$ is tangent to $\Gamma$ itself at isolated points only;
moreover, at the two sides on $\Gamma$ of each of these points,
$v$ points to opposite sides of $\Gamma$ on $\partial M$.
\end{itemize}

To graphically illustrate the situation, we introduce some terminology
that we will repeatedly employ in the rest of the paper:
\begin{itemize}
\item We call \emph{in-region} (respectively, \emph{out-region})
the union of the components of $(\partial M)\setminus\Gamma$
on which $v$ points towards the interior (respectively, the exterior) of $M$;
\item If $A$ is a point of $\Gamma$ we will say that
$A$ is \emph{concave}
if at $A$ the field $v$ points from the out-region to the in-region,
and \emph{convex} if it points
from the in-region to the out-region; this terminology is borrowed from~\cite{LNM} and
is motivated by the shape of the orbits of $v$ near $A$, see Fig.~\ref{concave/convex:fig};
\begin{figure}
    \begin{center}
    \includegraphics[scale=.6]{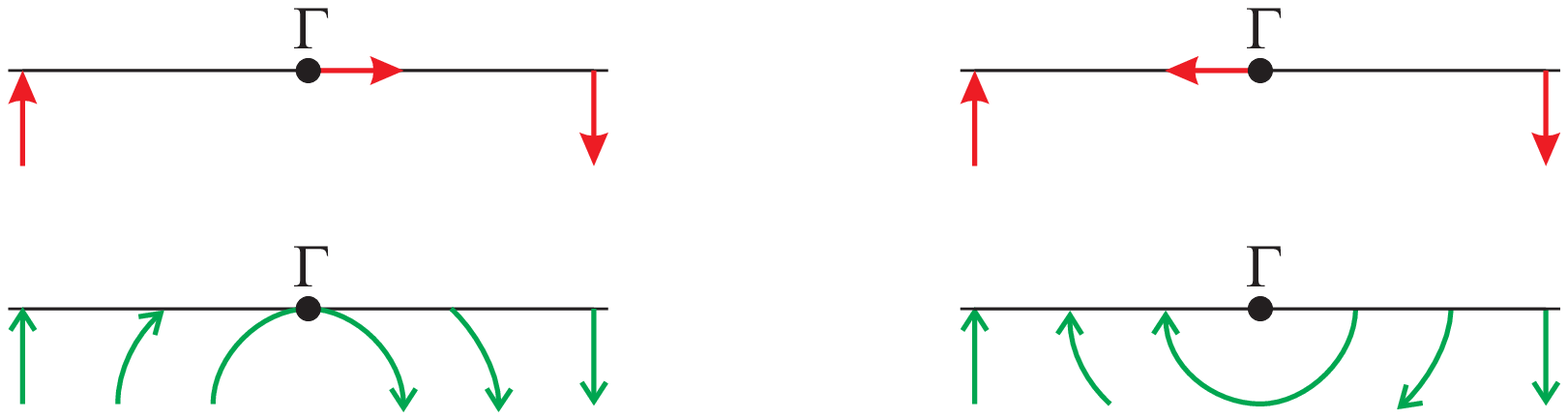}
    \end{center}
\mycap{Orbits of $v$ near a concave (left) and near a convex (right) point of $\Gamma$.
All pictures represent a cross section transverse to $\Gamma$. The top pictures
show $v$, the bottom ones show its orbits.\label{concave/convex:fig}}
\end{figure}
\item A point $A$ of $\Gamma$ at which $v$ is tangent to $\Gamma$
will be termed \emph{transition} point;
as one easily sees, there are up to
diffeomorphism only $2$ local models for the field $v$ near $A$,
as shown in Fig.~\ref{transition/types:fig}.
\begin{figure}
    \begin{center}
    \includegraphics[scale=.5]{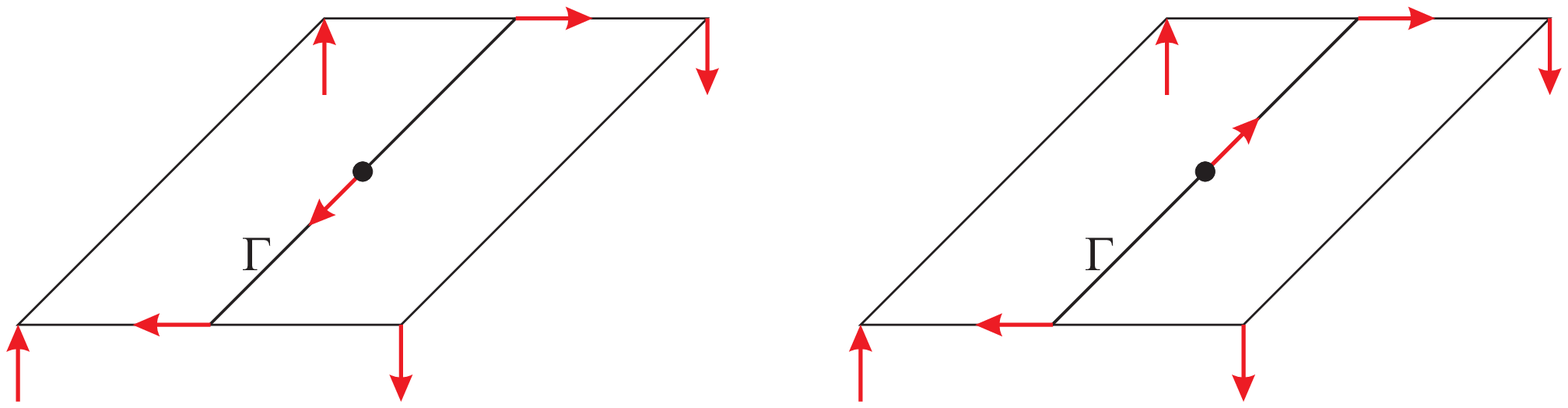}
    \end{center}
\mycap{Types of transition points; on the left $v$ points from the concave
to the convex portion of $\Gamma$, on the right from the convex to the concave portion of $\Gamma$;
note that mirror images in $3$-space of these configurations should also be taken into account (namely,
the figures are unoriented).\label{transition/types:fig}}
\end{figure}
\end{itemize}

The next result records obvious facts and two less obvious ones:

\begin{prop}\label{trans:fate:prop}
Let $A$ be a point of $\partial M$. Then, depending on where $A$ lies,
the orbit of $v$ through $A$ extends as follows:
\begin{center}
\begin{tabular}{l|l}
$A$ in the in-region & Only forward \\
$A$ in the out-region & Only backward \\
$A$ a concave point & Both forward and backward \\
$A$ a convex point & Neither forward nor backward \\
$A$ a concave-to-convex transition point  & Only backward\\
$A$ a convex-to-concave transition point & Only forward
\end{tabular}
\end{center}
\end{prop}

\begin{proof}
The result is evident except for orbits through the transition points.
To deal with them
we first analyze what the orbits would be if $v$ were projected
to $\partial M$, which we do in Fig.~\ref{transition/orbits/first:fig}.
\begin{figure}
    \begin{center}
    \includegraphics[scale=.5]{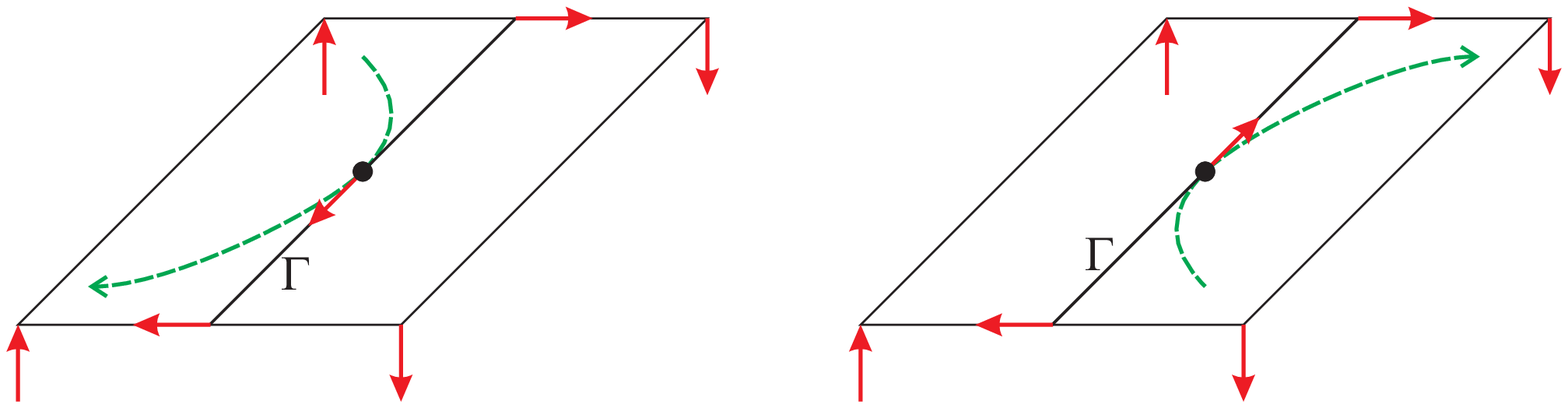}
    \end{center}
\mycap{Orbits through the transition points for the field obtained by projecting $v$
to a vector field tangent to $\partial M$.\label{transition/orbits/first:fig}}
\end{figure}
The picture shows that at the concave-to-convex transition points the orbit
of the projection of $v$ lies in the out-region, which implies that the
orbit of $v$ extends backward but not forward, while at the
convex-to-concave transition points the opposite happens.
\end{proof}

From now on an orbit of $v$ reaching a concave-to-convex transition point or
leaving from a convex-to-concave transition point
will be termed \emph{transition orbit}.

\subsection{Streams}
Our main aim in this paper is to provide a combinatorial presentation of the set of generic flows on $3$-manifolds
up to homotopy fixed on the boundary, but to achieve this aim we first need to somewhat restrict the
class of flows we consider and the equivalence relation on them.
Informally, we call \emph{stream} on a $3$-manifold $M$ a vector field $v$
satisfying (G1) such that, in addition, all the orbits of $v$ start and end on $\partial M$, and
the orbits of $v$ tangent to $\partial M$ are generic with respect to each other.
More precisely, $v$ is a stream on $M$ if it satisfies the conditions (G1)-(G4), with:
\begin{itemize}
\item[(\textbf{G2})]
Every orbit of $v$ is either a single point (a convex point of $\Gamma$)
or a closed arc with both ends on $\partial M$;
\item[(\textbf{G3})]
The transition orbits are tangent to $\partial M$ at their transition point only.
\end{itemize}

For the next and last condition we note that if an arc of an orbit of $v$ has
ends $A$ and $B$ contained in the interior of $M$ then the parallel transport
along $v$ defines a linear bijection
from the tangent space to $M$ at $A$ to that at $B$. We then require the following:

\begin{itemize}
\item[(\textbf{G4})]
Each orbit of $v$ is tangent to $\partial M$ at two points at most;
if an orbit of $v$ is tangent to $\partial M$ at two points $A$ and $B$,
that necessarily are concave points of $\Gamma$ by conditions (G2) and (G3),
then the tangent directions to $\Gamma$ at $A$ and at $B$ are transverse to
each other under the bijection defined by the parallel transport along $v$.
\end{itemize}

This last condition is illustrated in Fig.~\ref{transverse:fig}.
\begin{figure}
    \begin{center}
    \includegraphics[scale=.6]{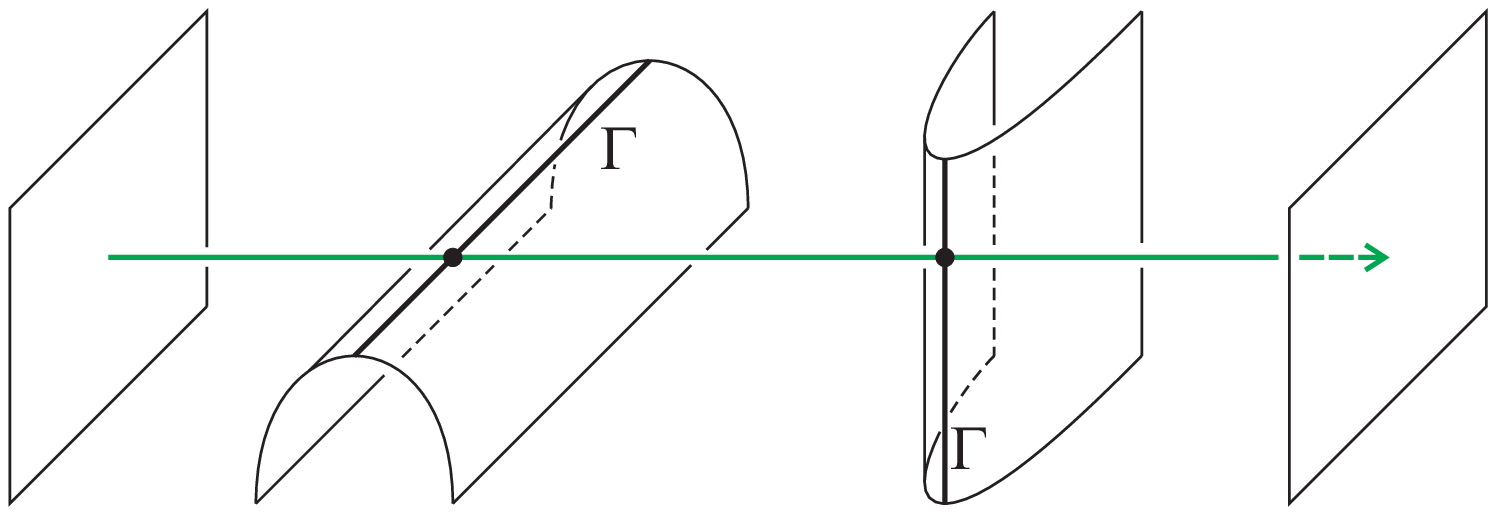}
    \end{center}
\mycap{If an orbit of $v$ is tangent to $\partial M$ at two points of $\Gamma$,
the two involved arcs of $\Gamma$ are transverse to each other under the parallel transport along $v$.\label{transverse:fig}}
\end{figure}
We will henceforth denote by $\calF_0^*$ the set of pairs $(M,v)$ with
$M$ an oriented, compact, connected $3$-manifold and $v$ a stream on $M$,
up to diffeomorphism.

\subsection{Stream-spines}
We now introduce the objects that will eventually be shown to be the combinatorial
counterparts of streams on smooth oriented $3$-manifolds.
As above, stating all the requirements takes some time and involves some new terminology.
We will then stepwise introduce 3 conditions
(S1), ({S2}), ({S3}) for a compact and connected $2$-dimensional polyhedron $P$, the combination of
which will constitute the definition of a \emph{stream-spine}. We begin with the following:

\begin{itemize}
\item[(\textbf{S1})] A neighbourhood of each point of $P$ is homeomorphic to one of the $5$ models
of Fig.~\ref{local/stream/spine:fig}.
\begin{figure}
    \begin{center}
    \includegraphics[scale=.6]{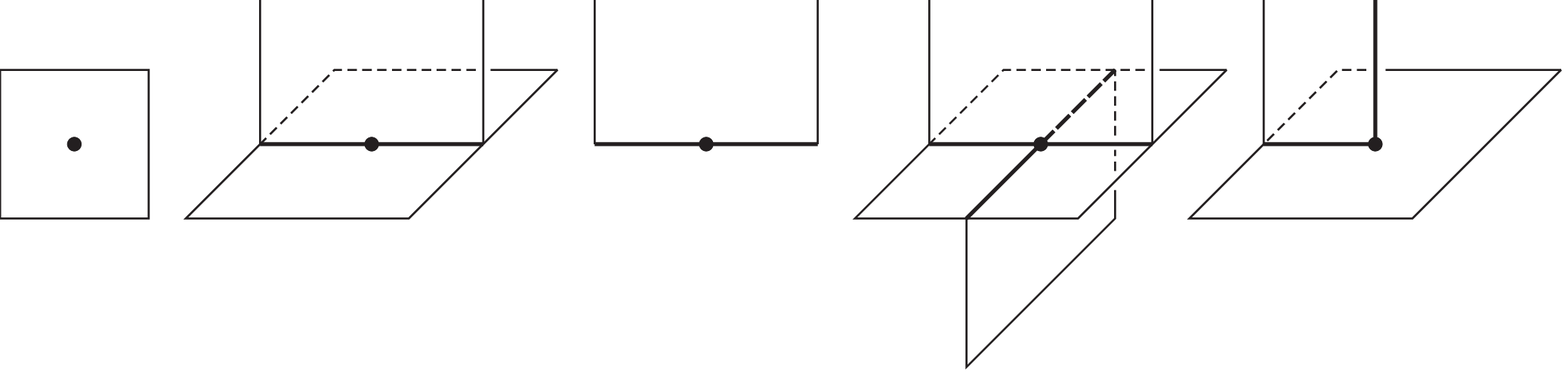}
    \end{center}
\mycap{Local aspect of a stream-spine.\label{local/stream/spine:fig}}
\end{figure}
\end{itemize}
This condition implies that $P$ consists of:
\begin{enumerate}
\item Some open surfaces, called \emph{regions},
each having a closure in $P$ which is a compact surface with possibly immersed boundary;
\item Some \emph{triple lines}, to which three regions are locally incident;
\item Some \emph{single lines}, to which only one region is locally  incident;
\item A finite number of points, called \emph{vertices}, to which six regions are locally incident;
\item A finite number of points, called \emph{spikes}, to which both a triple and a single line are incident.
\end{enumerate}

We note that a polyhedron satisfying condition (S1) is simple according
to Matveev~\cite{Matveev:AAM}, but not almost-special if single lines exist.
Our next condition was first introduced in~\cite{BP:Manuscripta};
to state it we define a \emph{screw-orientation} along an arc of triple line of $P$ as an orientation
of the arc together with a cyclic ordering of the three germs of regions of $P$ incident
to the arc, viewed up simultaneous reversal of both, as in
Fig.~\ref{screw:fig}-left.
\begin{figure}
    \begin{center}
    \includegraphics[scale=.6]{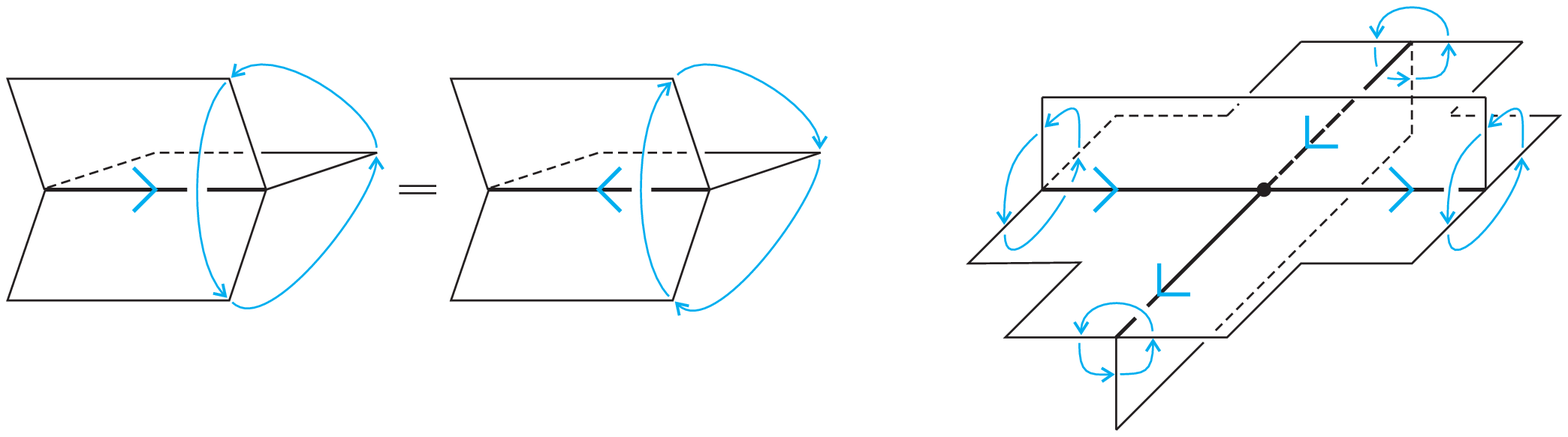}
    \end{center}
\mycap{Screw orientation along a triple line, and compatibility at vertices.\label{screw:fig}}
\end{figure}

\begin{itemize}
\item[(\textbf{S2})] Along each triple line of $P$ a screw-orientation is defined in such
a way that at each vertex the screw-orientations are as in Fig.~\ref{screw:fig}-right.
\end{itemize}

We now give the last condition of the definition of stream-spine:
\begin{itemize}
\item[(\textbf{S3})]
Each region of $P$ is orientable, and it is endowed with a specific orientation, in such a way that no
triple line is induced three times the same orientation by the regions incident to it.
\end{itemize}

We will say that two stream-spines are \emph{isomorphic} if they are related by a PL homeomorphism
respecting the screw-orientations along triple lines and the orientations of the regions, and
we will denote by $\calS_0$ the set of all stream-spines up to isomorphism.

\subsection{Stream carried by a stream-spine}

In this subsection we will show that each stream-spine uniquely defines an oriented smooth
manifold and a stream on it. To begin we take a compact polyhedron $P$ satisfying condition
({S1}) of the definition of stream-spine, namely locally appearing as in
Fig.~\ref{local/stream/spine:fig}. We will say that an embedding of $P$ in a
$3$-manifold $M$ is \emph{branched} if the following happens upon identifying $P$ with
its image in $M$ (see Fig.~\ref{branched/embedding:fig}):
\begin{figure}
    \begin{center}
    \includegraphics[scale=.6]{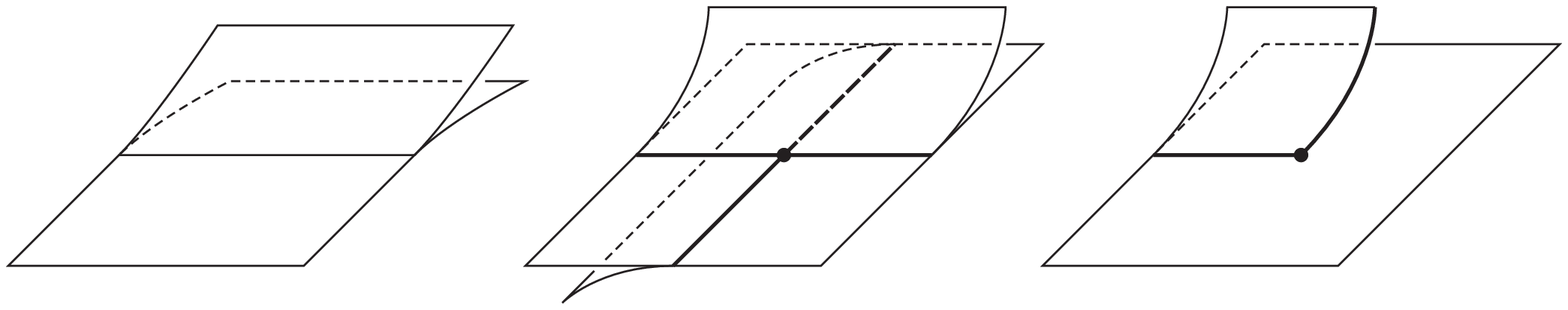}
    \end{center}
\mycap{A polyhedron locally as in Fig.~\ref{local/stream/spine:fig}
sitting in a branched fashion in a $3$-manifold
(any mirror image in $3$-space of these figures is also allowed).
\label{branched/embedding:fig}}
\end{figure}
\begin{itemize}
\item Each region of $P$ has a well-defined
tangent plane at every point;
\item If a point $A$ of $P$ lies on a triple line but is neither a vertex nor a spike,
the tangent planes at $A$ to the $3$ regions of $P$ locally incident to $A$ coincide,
and not all the $3$ regions of $P$ locally project to one and the same half-plane of this tangent plane;
\item At a vertex $A$ of $P$ the tangent planes at $A$ to the $6$ regions of $P$ locally incident to $A$ coincide;
\item At a spike $A$ of $P$ the tangent planes at $A$ to the $2$ regions of $P$ locally incident to $A$ coincide.
\end{itemize}

\begin{prop}\label{spine:to:stream:prop}
To any  stream-spine $P$ there correspond a smooth compact oriented $3$-manifold $M$ and
a stream $v$ on $M$ such that $P$ embeds in a branched fashion in $M$, the field
$v$ is everywhere positively transversal to $P$, and $M$ is homeomorphic to a regular
neighbourhood of $P$ in $M$; the pair $(M,v)$ is well-defined up to oriented diffeomorphism, therefore
setting $\varphi(P)=(M,v)$ one gets a well-defined map $\varphi_0^*:\calS_0\to\calF_0^*$.
\end{prop}

\begin{proof}
Our first task is to show that $P$ thickens in a PL sense to a well-defined
oriented manifold $M$ (later we will need to describe a smooth structure for
$M$ and the field $v$). This argument extends that of~\cite{BP:Manuscripta}.
Let us denote by $U$ a regular neighbourhood in $P$ of the union
of the triple lines. We observe that $U$ can be seen as a union of fragments as in
Fig.~\ref{thicken:fig}-top, that we thicken
\begin{figure}
    \begin{center}
    \includegraphics[scale=.6]{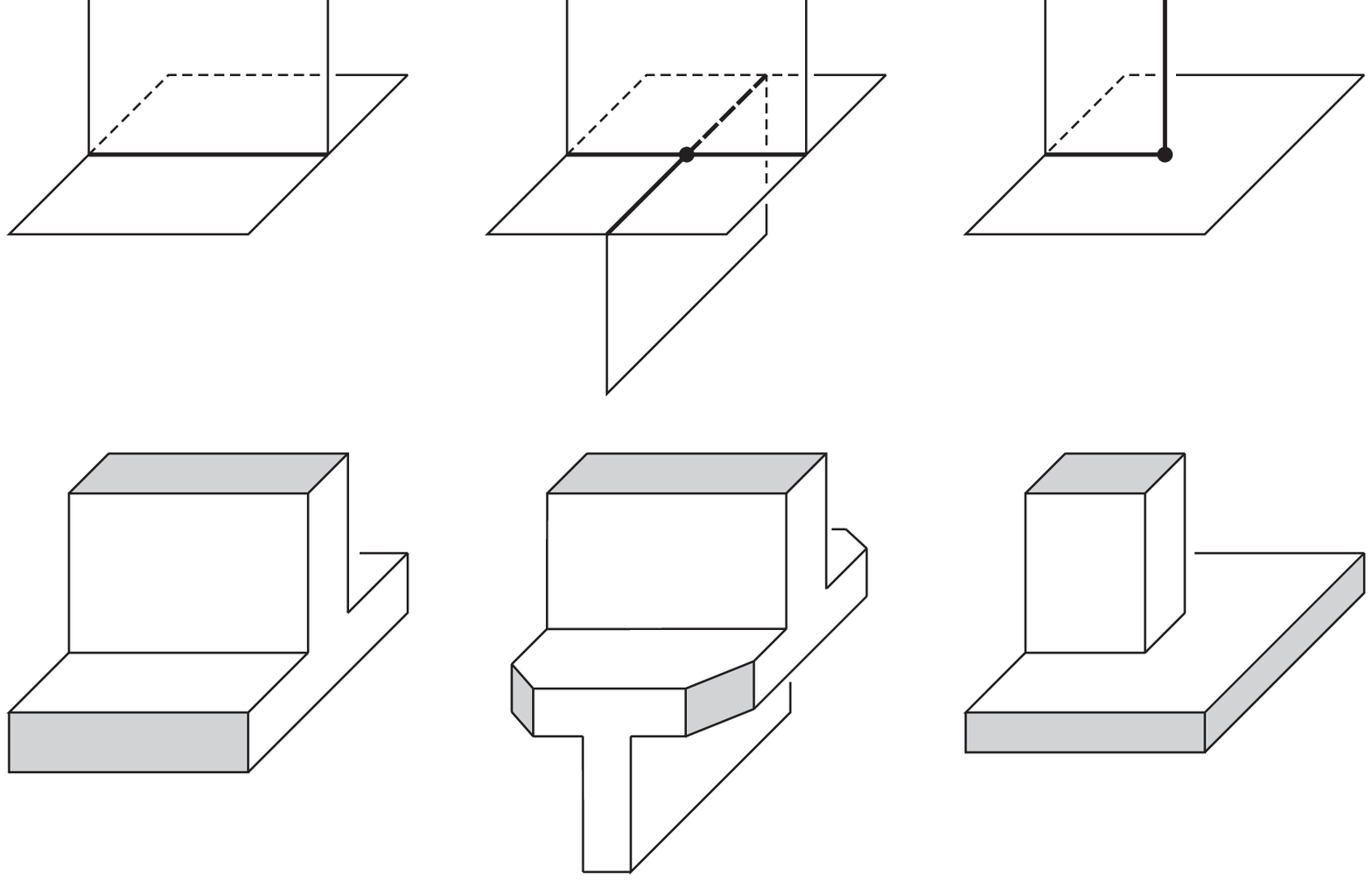}
    \end{center}
\mycap{Blocks obtained by thickening fragments
of a neighbourhood of the union of the triple lines.\label{thicken:fig}}
\end{figure}
as shown in the bottom part of the same figure, giving each block the
orientation such that the screw-orientations along the
portions of triple lines of $P$ within each block are positive.
Note that on the boundary of each block there are some T-shaped regions and that
some rectangles are highlighted. Following the
way $U$ is reassembled from the fragments into which it was decomposed, we can now
assemble the blocks by gluing together the T's on their boundary.
(Note that the gluing between two T's need not identify the vertical legs to each other,
so each T should actually be thought of as a Y: the three legs play symmetric r\^oles.)
Since the gluings automatically reverse the orientation, the result is an
oriented manifold, on the boundary of which we have some highlighted strips,
each having the shape of a rectangle or of an annulus.
Now we turn to the closure in $P$ of the complement of $U$, that we denote by $S$.
Of course $S$ is a surface with boundary, and on $\partial S$ we can highlight the
arcs and circles shared with $U$.  (The rest of $\partial S$ consists of arcs lying on single lines of $P$.)
We then take the product $S\times I$ ---this is a crucial choice that will be discussed below---
and note that the highlighted arcs and circles on $\partial S$ give
highlighted rectangles and annuli on $\partial (S\times I)$. We are only left to glue these
rectangles and annuli to those on the boundary of the assembled blocks,
respecting the way $S$ is glued to $U$ and making sure the orientation is reversed.
The result is the required manifold $M$.

We must now explain how to smoothen $M$ and how to choose the stream $v$. Away from the triple and single lines of $P$
the manifold $M$ is the product $S\times I$ with $S$ a surface, so it is sufficient to
smoothen $S$ and to define $v$ to be parallel to the $I$ factor and positively transversal to $S$.
(This justifies our choice of thickening $S$ as a trivial rather than some other  $I$-bundle.)
Along the triple and single lines of $P$ we extend this construction as suggested
in a cross-section in Fig.~\ref{smooth/thick:fig}.
\begin{figure}
    \begin{center}
    \includegraphics[scale=.6]{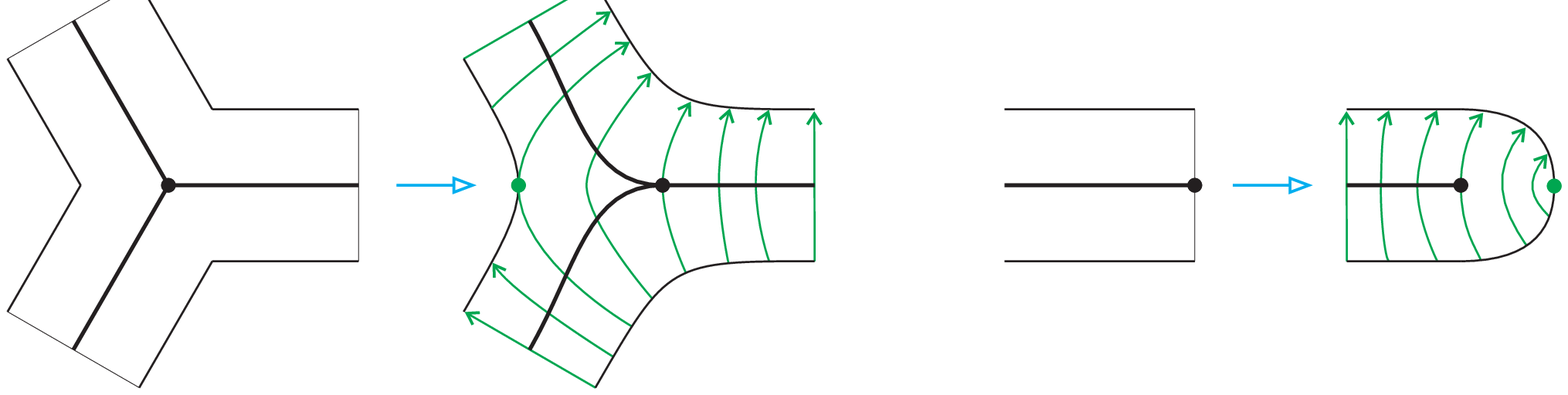}
    \end{center}
\mycap{The stream along triple and single lines.\label{smooth/thick:fig}}
\end{figure}
Note that a triple line of $P$ gives rise to a concave tangency
line of $v$ to $\partial M$, and that a single line of $P$ gives rise to a
convex tangency line. To conclude we must illustrate the extension of the construction of $v$ near
vertices and near spikes, which we do in two examples in Fig.~\ref{vert/spike/thick:fig}.
\begin{figure}
    \begin{center}
    \includegraphics[scale=.6]{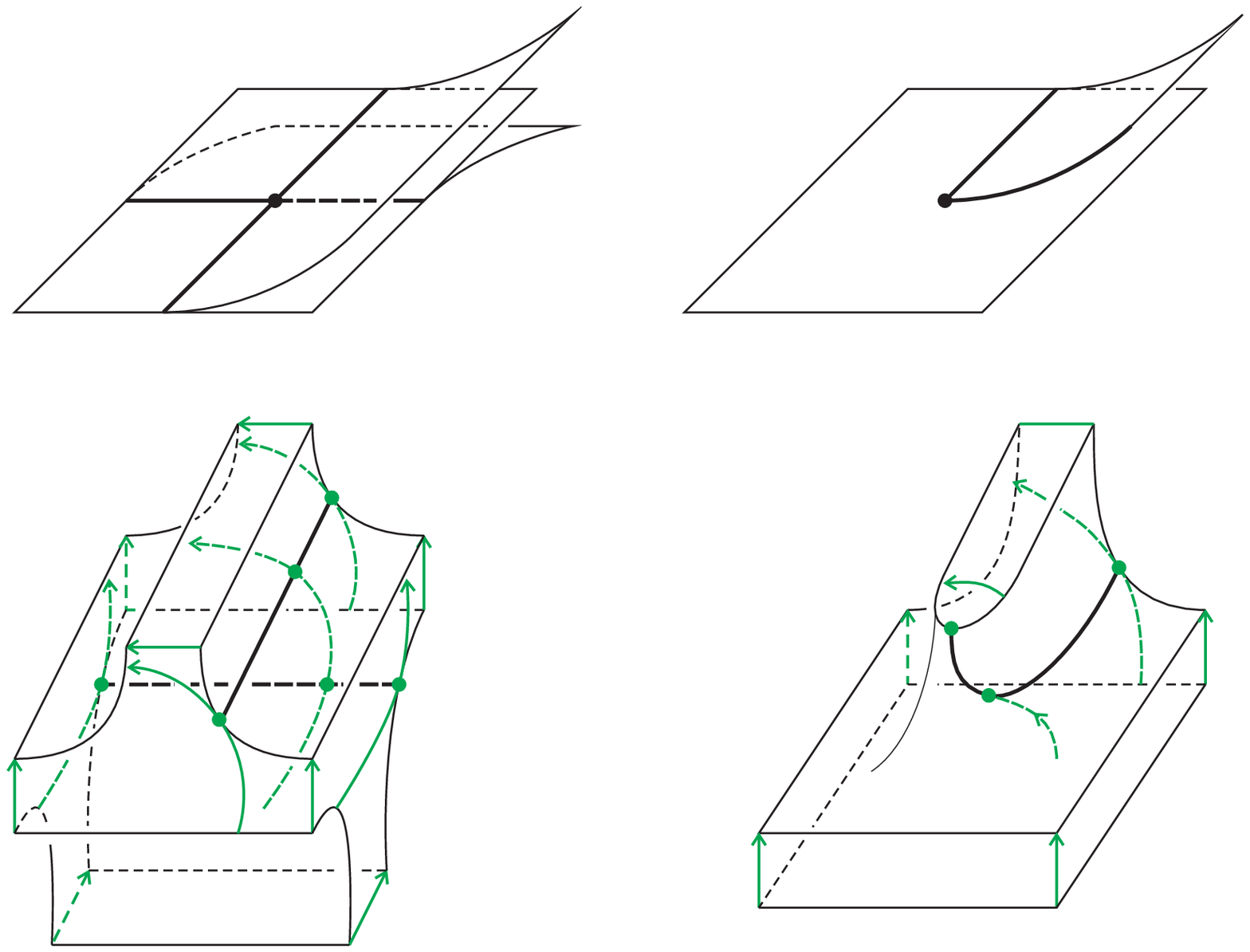}
    \end{center}
\mycap{Stream carried by a stream-spine near a vertex and near a spike.\label{vert/spike/thick:fig}}
\end{figure}
In the figure we represent $v$ by showing some of its orbits. Note that:
\begin{itemize}
\item In both cases the local configurations of $v$ near $\partial M$ are as in condition (G1) of the definition of stream;
\item The orbits of $v$ are closed arcs or points, as in condition (G2);
\item To a vertex of $P$ there corresponds an orbit of $v$ that is tangent to $\partial M$ at two points,
in a concave fashion and respecting the transversality condition (G4);
\item To a spike of $P$ there corresponds a transition orbit of $v$ satisfying condition (G3).
\end{itemize}
This shows that $v$ is indeed a stream on $M$. Since che construction of $(M,v)$ is uniquely
determined by $P$, the proof is complete.
\end{proof}

\subsection{The in-backward and the out-forward\\
stream-spines of a stream}
In this subsection we prove that the construction of Proposition~\ref{spine:to:stream:prop}
can be reversed, namely that the map $\varphi_0^*:\calS_0\to\calF_0^*$ is bijective.
More exactly, we will see that the topological
construction has two inverses that are
equivalent to each other ---but not obviously so. If $v$ is a stream on a $3$-manifold $M$ we define:
\begin{itemize}
\item The \emph{in-backward} polyhedron associated to $(M,v)$
as the closure of the in-region of $\partial M$
union the set of all points $A$ such that
there is an orbit of $v$ going from $A$ to a concave or transition point of $\partial M$;
\item The \emph{out-forward} polyhedron associated to $(M,v)$
as the closure of the out-region of $\partial M$ union the set of all points $A$ such that
there is an orbit of $v$ going from a concave or transition point of $\partial M$ to $A$.
\end{itemize}

\begin{prop}\label{stream:to:spine:prop}\ \\ \vspace{-.4cm}
\begin{itemize}
\item Let $v$ be a stream on $M$. Then the
in-backward and out-forward polyhedra associated to $(M,v)$
satisfy condition (S1) of the definition of stream-spine;
moreover each of their regions shares some point with the in-region or with the out-region of $\partial M$, and
it can be oriented so that at these points the field $v$ is positive transversal to it;
with this orientation on each region, the
in-backward and out-forward polyhedra associated to $(M,v)$ are stream-spines, they
are isomorphic to each other and
via Proposition~\ref{spine:to:stream:prop} they both define
the pair $(M,v)$.
\item If $P$ is a stream-spine and $(M,v)$ is the associated manifold-stream pair
as in Proposition~\ref{spine:to:stream:prop}, then the
in-backward and out-forward polyhedra associated to $(M,v)$
are isomorphic to $P$.
\end{itemize}
\end{prop}

\begin{proof}
Most of the assertions are easy, so we confine ourselves to the main points.
It is first of all obvious that away from the special orbits of $v$ as in conditions
(G3) and (G4) the concave tangency lines of $v$ to $\partial M$ generate
triple lines in the in-backward and out-forward polyhedra associated to $(M,v)$, while convex
tangency lines generate single lines. Moreover, if from a stream-spine $P$
we go to $(M,v)$ and then to the associated
in-backward and out-forward polyhedra, away from the
vertices and spikes of $P$ we see that these polyhedra
are naturally isomorphic to $P$, as shown in a cross-section in Fig.~\ref{P/to/M/to/IO:fig}.
\begin{figure}
    \begin{center}
    \includegraphics[scale=.6]{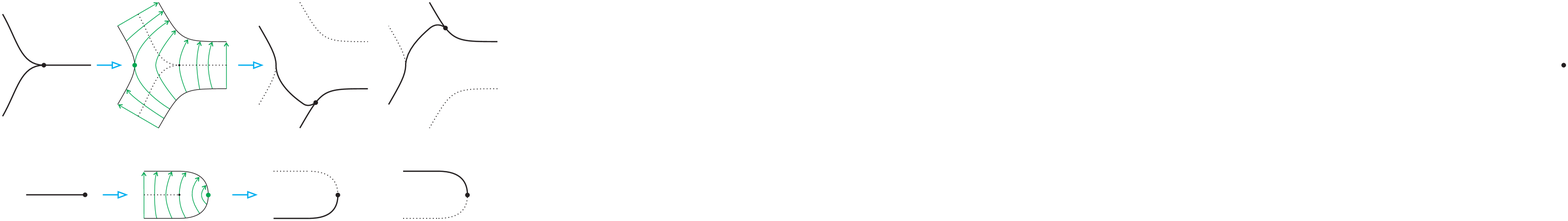}
    \end{center}
\mycap{From a stream-spine to a manifold-stream pair to
its in-backward and out-forward polyhedra.
Cross-section away from the vertices and spikes of the stream-spine and away from
the special orbits of the stream.\label{P/to/M/to/IO:fig}}
\end{figure}

The fact that an orbit of $v$ as in condition (G4) generates a vertex in
the in-backward and out-forward polyhedra associated to $(M,v)$
was already shown in~\cite{LNM}, but we reproduce the argument here
for the sake of completeness, showing in Fig.~\ref{G4/to/vertex:fig}-left,
top and bottom,
\begin{figure}
    \begin{center}
    \includegraphics[scale=.6]{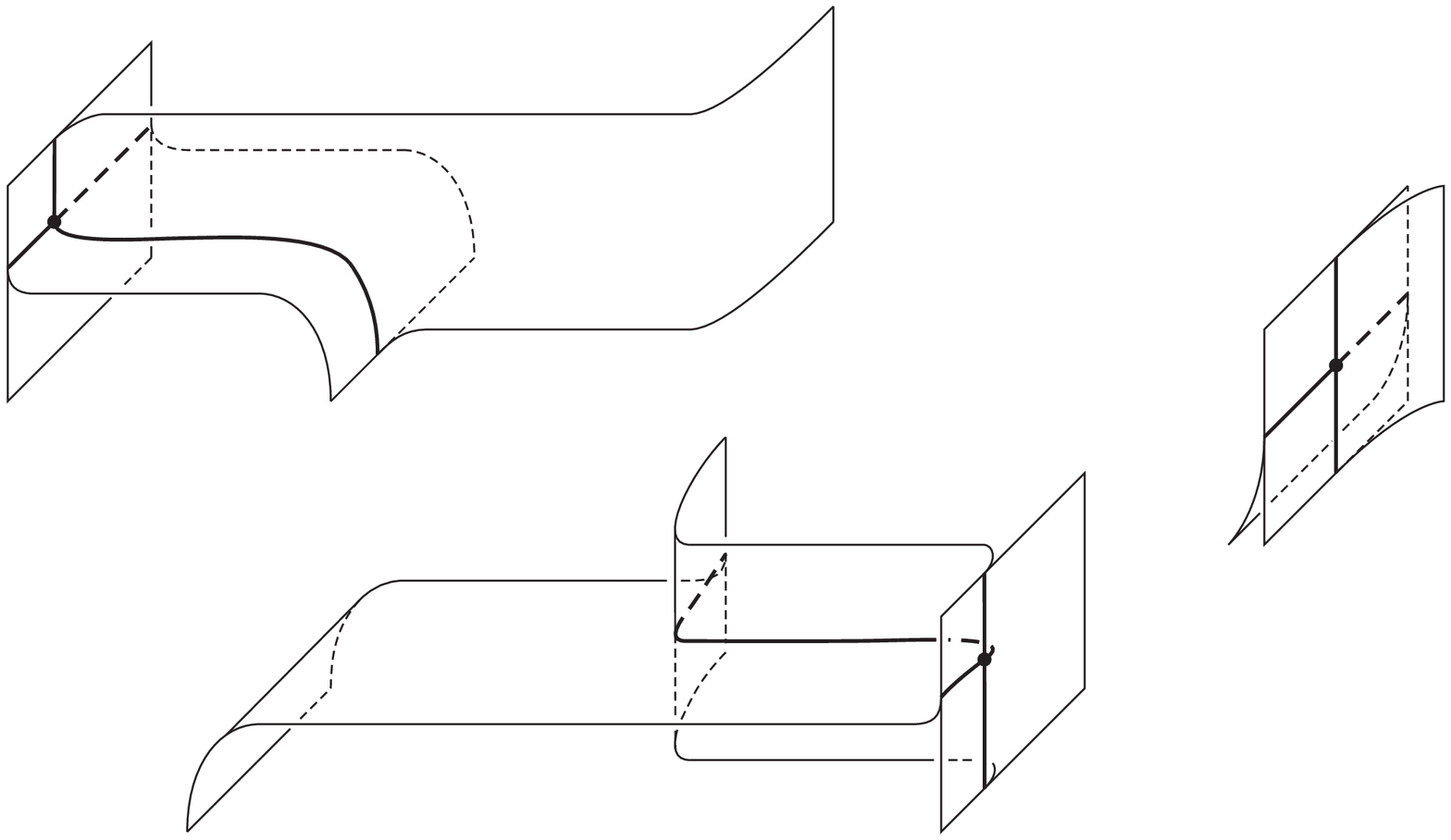}
    \end{center}
\mycap{An orbit of a stream doubly tangent to the boundary in a concave fashion
generates a vertex in the in-backward and in the out-forward stream-spines.\label{G4/to/vertex:fig}}
\end{figure}
the in-backward and the out-forward spines near the orbit of Fig.~\ref{transverse:fig}.
Both these spines are locally isomorphic to the stream-spine shown on the right, to
which Proposition~\ref{spine:to:stream:prop} associates precisely an orbit as in Fig.~\ref{transverse:fig}.

We are left to deal with transition points and with spikes.
Let us concentrate on a
concave-to-convex transition point as in Fig.~\ref{transition/types:fig}-left,
but mirrored and rotated in $3$-space for convenience.
In this case the transition orbit
extends backward (and not forward), and the
locally associated in-backward polyhedron is easy to describe,
which we do in Fig.~\ref{trans/spike:fig}-top.
\begin{figure}
    \begin{center}
    \includegraphics[scale=.6]{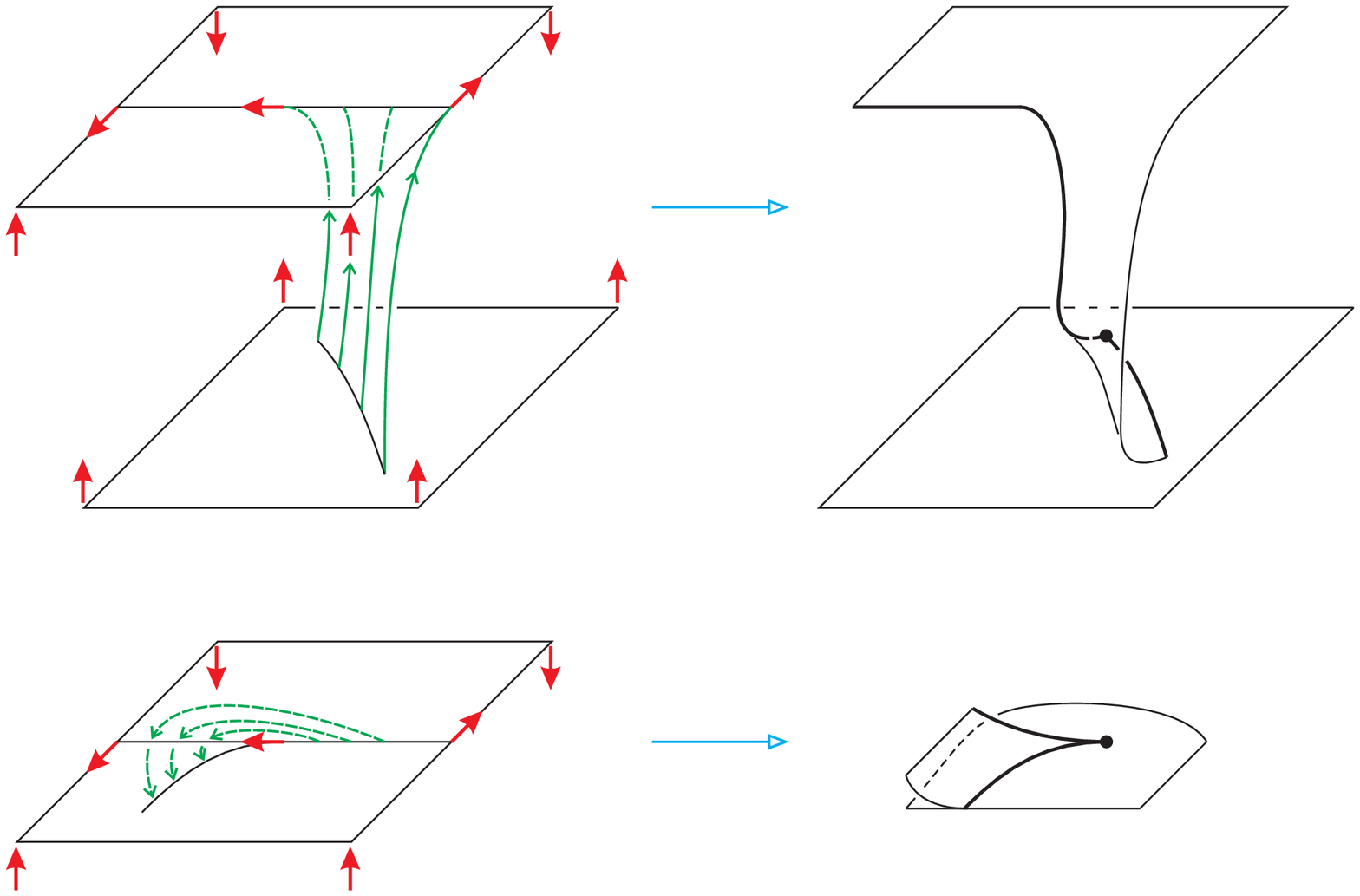}
    \end{center}
\mycap{From a transition point to a spike in the in-backward and in the out-forward associated polyhedra.\label{trans/spike:fig}}
\end{figure}
The out-forward polyhedron is instead slightly more complicated to understand, since the orbits
of $v$ starting from the concave line near the transition point finish on points close to the transition one, as
illustrated in Fig.~\ref{trans/spike:fig}-bottom. The pictures shows that the spikes thus generated
are indeed locally the same. Moreover, the
concave-to-convex configuration of $v$ near $\partial M$
is precisely that generated by a spike as in Fig.~\ref{vert/spike/thick:fig}-right, which is again of the same type.
This concludes the proof.
\end{proof}

Combining Propositions~\ref{spine:to:stream:prop} and~\ref{stream:to:spine:prop} we get the following
main result of this section:

\begin{thm}\label{stream:homeo:thm}
The map $\varphi_0^*:\calS_0\to\calF_0^*$ from the set of stream-spines up to isomorphism
to the set of streams on $3$-manifolds up to diffeomorphism is a bijection.
\end{thm}

\section{Stream-homotopy and\\ sliding moves on stream-spines}

In this section we consider a natural equivalence relation on streams, and we translate it into combinatorial moves on stream-spines.

\subsection{Elementary homotopy catastrophes}

Let $M$ be an oriented $3$-manifold with non-empty boundary. On the set $\calF_0^*$ of streams on $M$ we define
\emph{stream-homotopy} as the equivalence relation of homotopy through vector fields with
fixed configuration on $\partial M$ and all orbits homeomorphic to closed intervals or
to points. We then define $\calF_0$ as the quotient of $\calF_0^*$ under the
equivalence relation of stream-homotopy.
The next result shows how to factor this relation into easier ones:

\begin{prop}\label{catastrophes:prop}
Stream-homotopy is generated by isotopy and by the
elementary moves shown in
Figg.~\ref{20/catastrophe:fig} to~\ref{catastrophe/trans/in:fig}.
%%%%%%%%%%%%%%%%%%%%%%%%%%%%%%%%%%%%%%%%%%%%%%%%%%%%%%%%%
\begin{figure}
    \begin{center}
    \includegraphics[scale=.6]{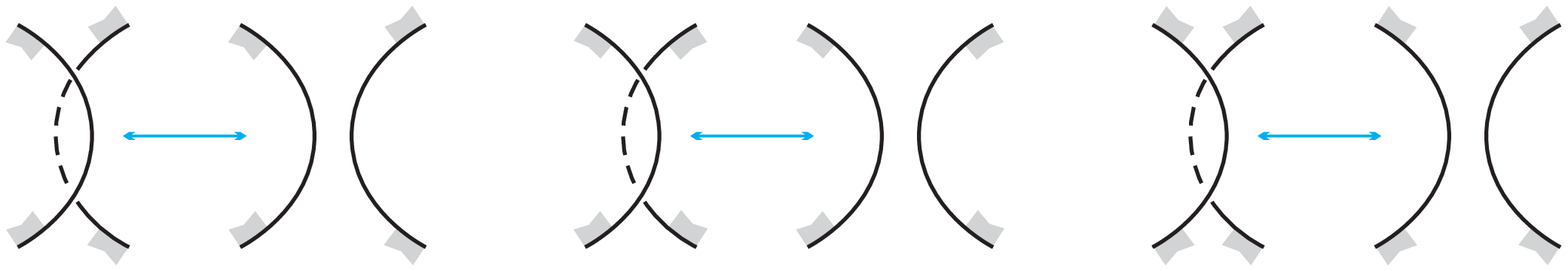}
    \end{center}
\mycap{Catastrophes corresponding to an orbit being twice concavely tangent to
the boundary but not in a transverse fashion. The pictures
show portions of the concave tangency line as seen looking in the direction
of the vector field, and they suggest to what part of it
the boundary of the manifold bends
\label{20/catastrophe:fig}}
\end{figure}
%%%%%%%%%%%%%%%%%%%%%%%%%%%%%%%%%%%%%%%%%%%%%%%%%%%%%%%%%
\begin{figure}
    \begin{center}
    \includegraphics[scale=.6]{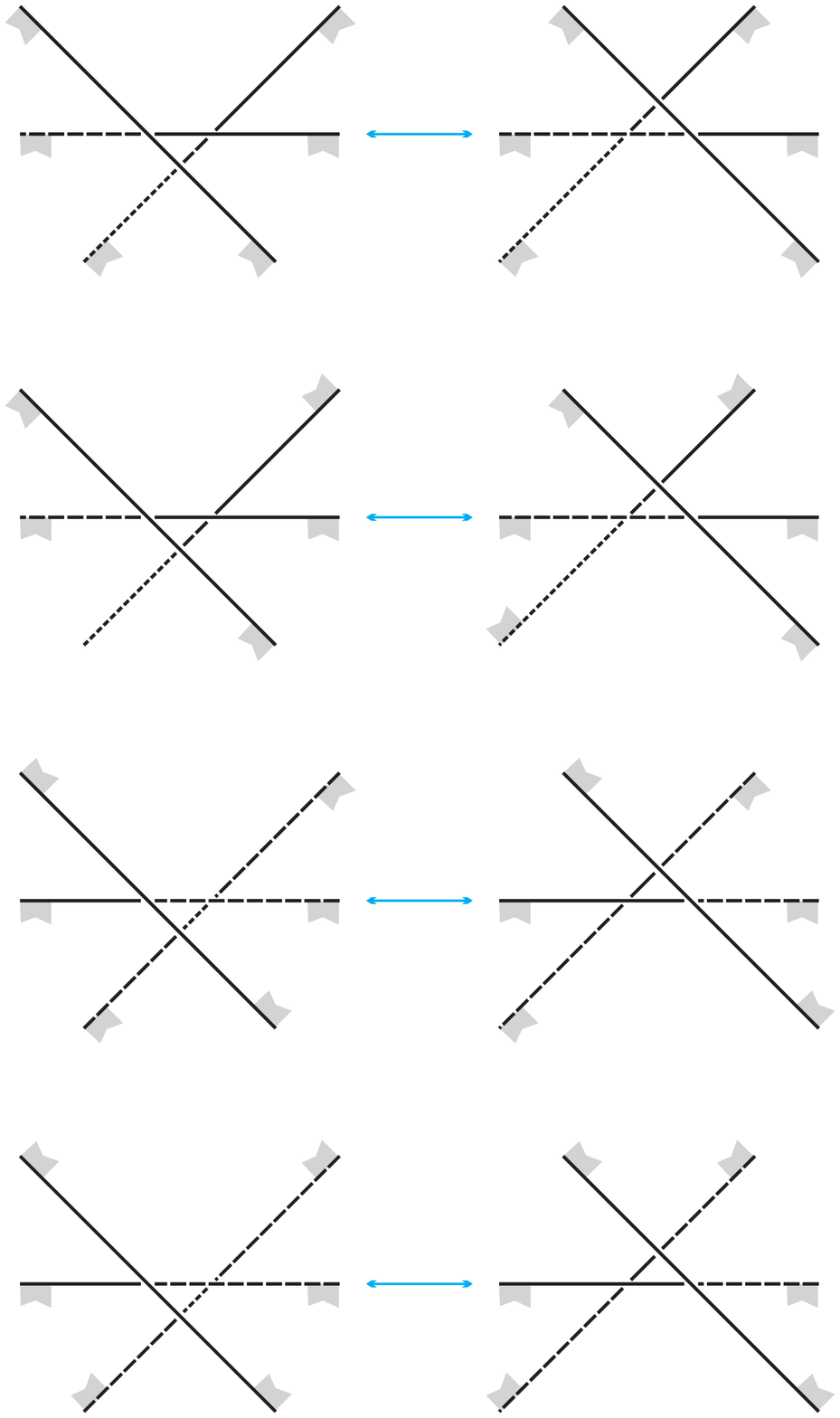}
    \end{center}
\mycap{Catastrophes
corresponding to an
orbit being thrice concavely tangent to the boundary in a transverse fashion.\label{32/catastrophe:fig}}
\end{figure}
%%%%%%%%%%%%%%%%%%%%%%%%%%%%%%%%%%%%%%%%%%%%%%%%%%%%%%%%%
\begin{figure}
    \begin{center}
    \includegraphics[scale=.6]{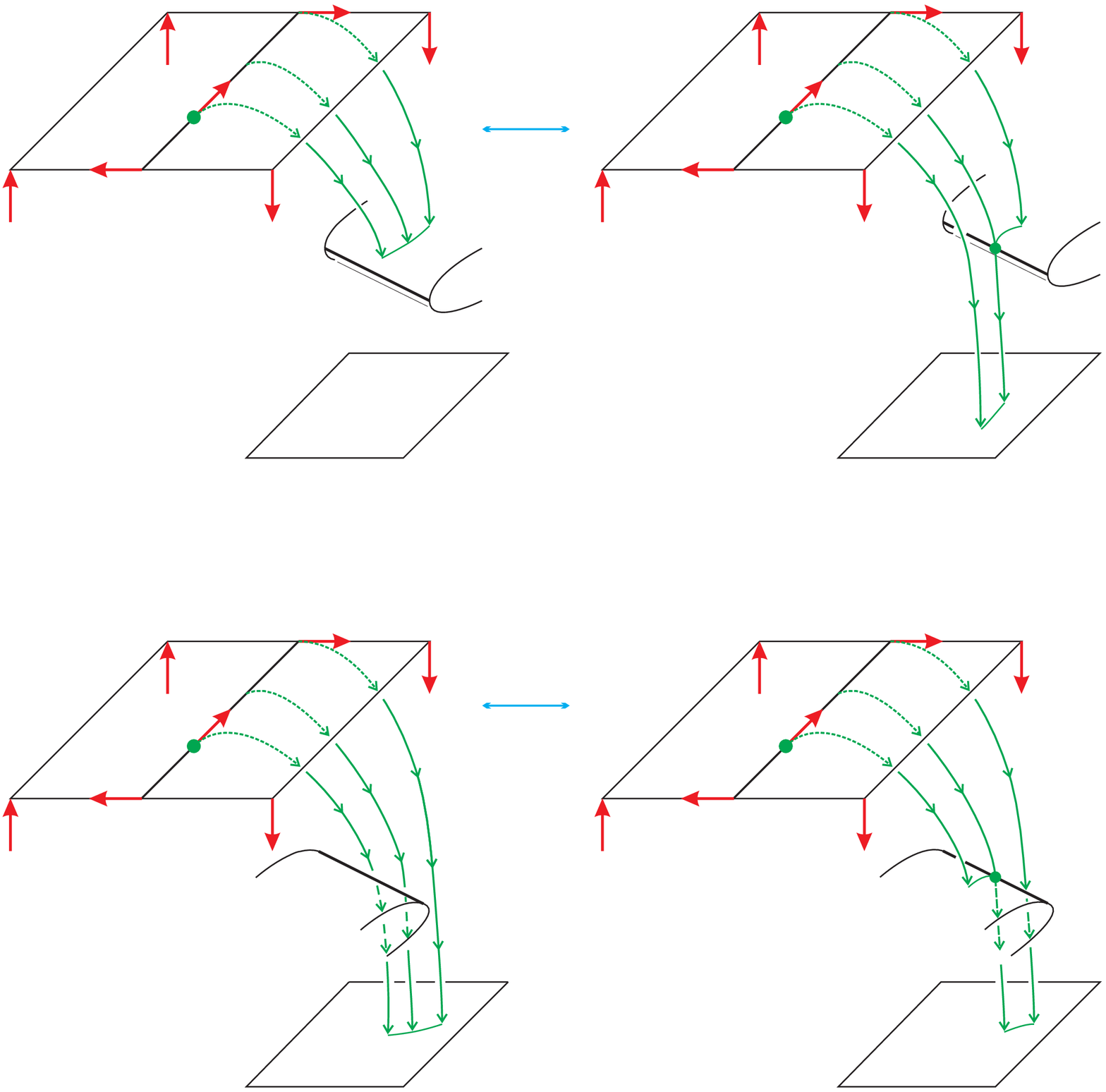}
    \end{center}
\mycap{Catastrophes corresponding to a transition orbit
being also once concavely tangent to the boundary, with an obvious transversality
condition.  These pictures refer to an incoming transition orbit,
but the analogue catastrophes involving outgoing transition orbits
must also be taken into account.\label{catastrophe/trans/in:fig}}
\end{figure}
%%%%%%%%%%%%%%%%%%%%%%%%%%%%%%%%%%%%%%%%%%%%%%%%%%%%%%%%%
\end{prop}

\begin{proof}
It is evident that taking a generic perturbation of a homotopy one only gets
the elementary catastrophes of the statement, plus perhaps finitely many times at which
an orbit starts and ends at transition points. We then only need to show that this type of catastrophe
can be generically avoided during a homotopy.
To do so we carefully analyze in Fig.~\ref{near/transition/in:fig}
\begin{figure}
    \begin{center}
    \includegraphics[scale=.6]{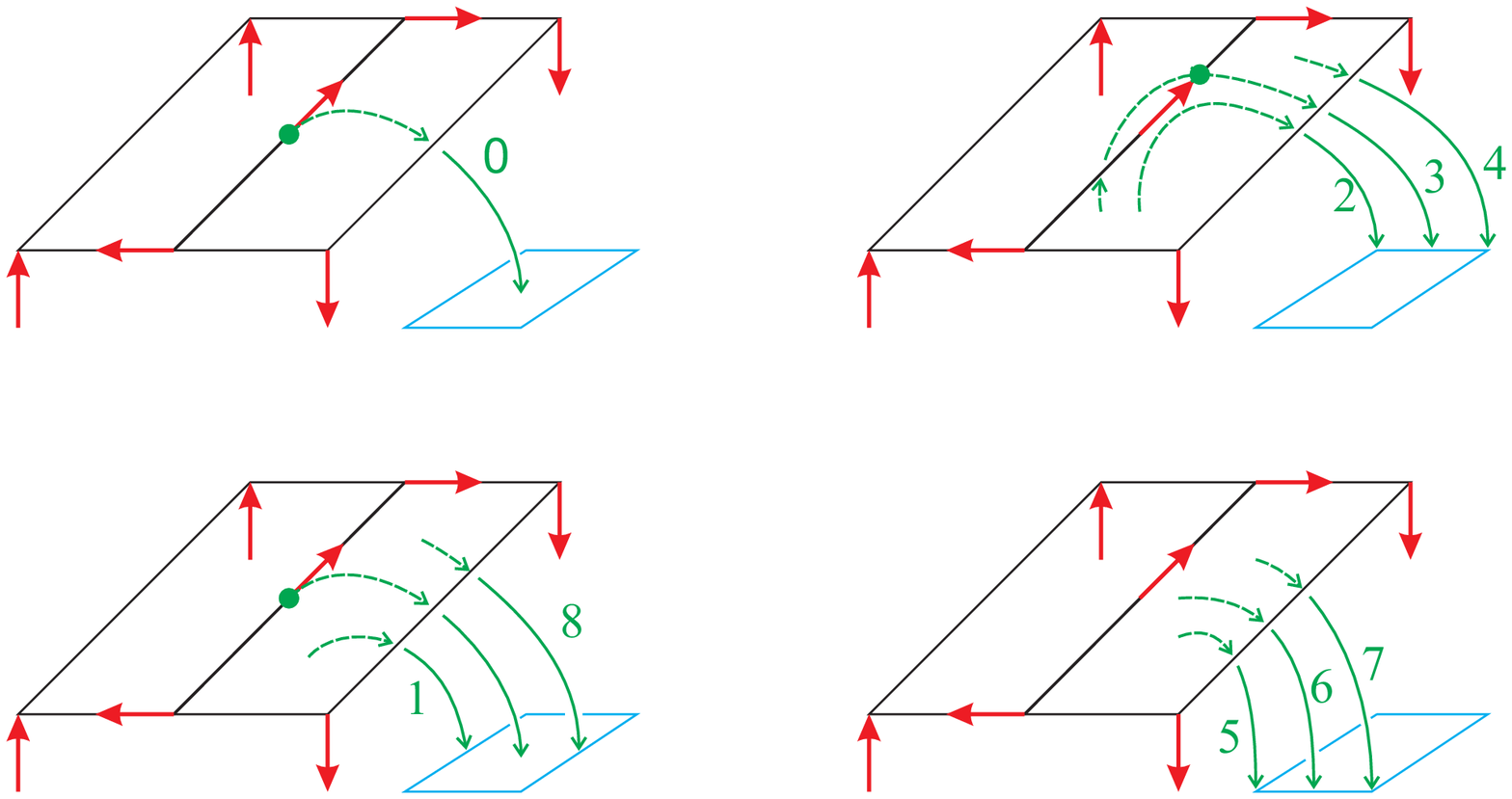}
    \end{center}
\mycap{Initial portions of orbits near an incoming transition orbit.\label{near/transition/in:fig}}
\end{figure}
the initial portions of the orbits close to an incoming transition orbit. In the type of catastrophe we want to avoid
we would have a concave-to-convex transition point $A$ such that the orbit through $A$
traces backward to, say, orbit $1$ just before the catastrophe, to orbit $0$ at the
catastrophe, and to orbit $8$ just after the catastrophe, with numbers as in
Fig.~\ref{near/transition/in:fig}. We can now modify the homotopy so that the orbit through $A$
traces back to either
\begin{itemize}
\item orbit 1, then 2, then 3, then 4, then 8, or
\item orbit 1, then 5, then 6, then 7, then 8.
\end{itemize}
Note that at $A$ with the first choice we obviously create a catastrophe as in
Fig.~\ref{catastrophe/trans/in:fig}, but for an outgoing transition orbit, while with the second
choice we do not create any catastrophe at $A$. On the other hand
at the starting point of orbit $0$ in Fig.~\ref{near/transition/in:fig}
we could create a
catastrophe as in Fig.~\ref{catastrophe/trans/in:fig} with one of the two choices and no
catastrophe with the other choice, but we cannot predict which is which. This shows that we can
always get rid of a doubly transition orbit
either at no cost or by inserting one catastrophe as in Fig.~\ref{catastrophe/trans/in:fig}.
\end{proof}

\subsection{Sliding moves on stream-spines}

In this subsection we introduce certain combinatorial moves on stream-spines. We do so
showing pictures and always meaning that the mirror images in $3$-space of the moves that we represent are
also allowed and named in the same way. Here comes the list; we call:
\begin{itemize}
\item \emph{Sliding $0\leftrightarrow 2$ move} any move as in Fig.~\ref{slide/20:fig};
\item \emph{Sliding $2\leftrightarrow 3$ move} any move as in Fig.~\ref{slide/32:fig};
\item \emph{Spike-sliding move} any move as in Fig.~\ref{slide/spike:fig};
\item \emph{Sliding move} any move of the types just described.
\end{itemize}
%%%%%%%%%%%%%%%%%%%%%%%%%%%%%%%%%%%%%%%%%%%%%%%%%%%%%%%%%
\begin{figure}
    \begin{center}
    \includegraphics[scale=.7]{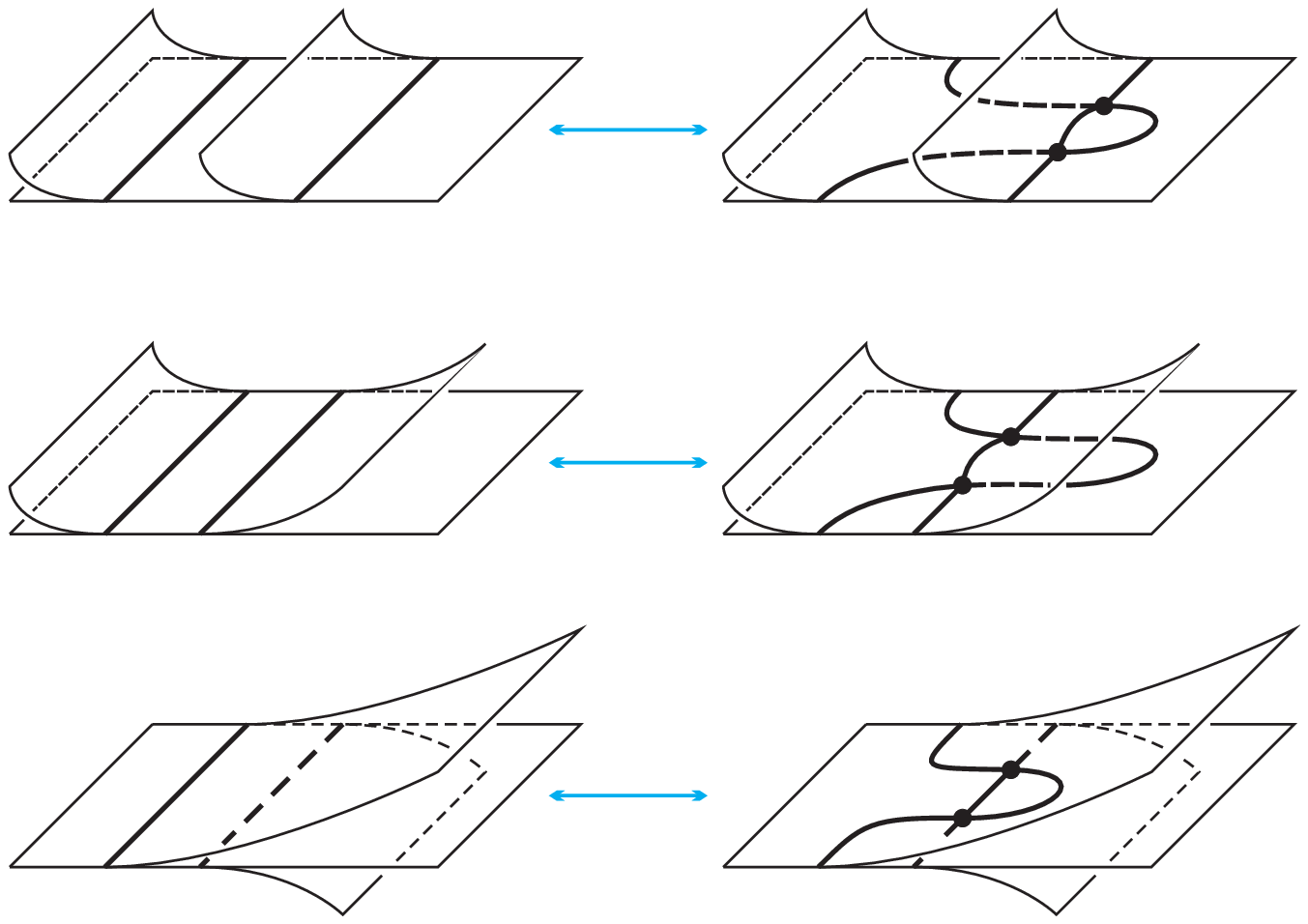}
    \end{center}
\mycap{The $0\leftrightarrow2$ sliding moves.\label{slide/20:fig}}
\end{figure}
%%%%%%%%%%%%%%%%%%%%%%%%%%%%%%%%%%%%%%%%%%%%%%%%%%%%%%%%%
\begin{figure}
    \begin{center}
    \includegraphics[scale=.7]{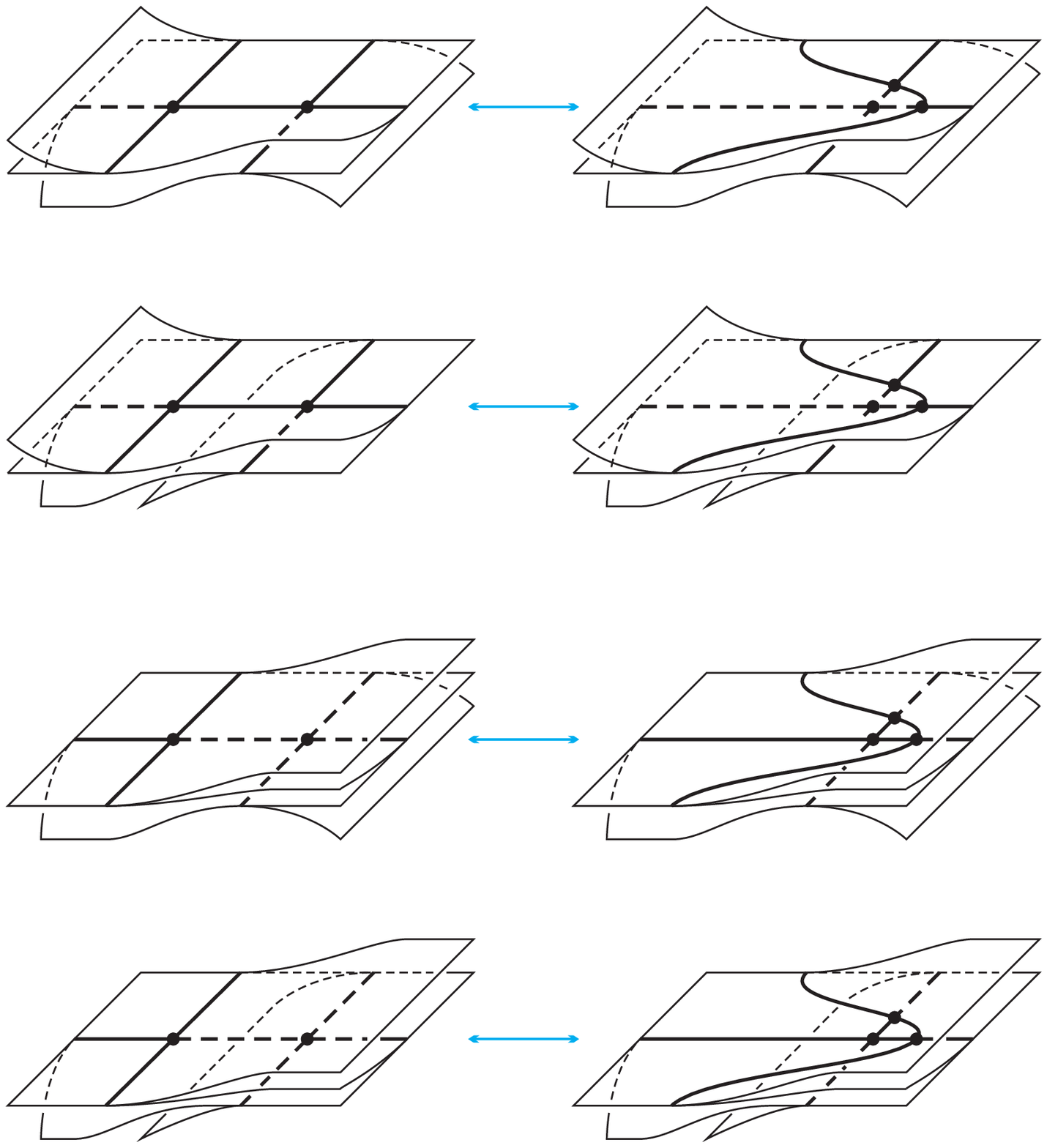}
    \end{center}
\mycap{The $2\leftrightarrow3$ sliding moves.\label{slide/32:fig}}
\end{figure}
%%%%%%%%%%%%%%%%%%%%%%%%%%%%%%%%%%%%%%%%%%%%%%%%%%%%%%%%%
\begin{figure}
    \begin{center}
    \includegraphics[scale=.7]{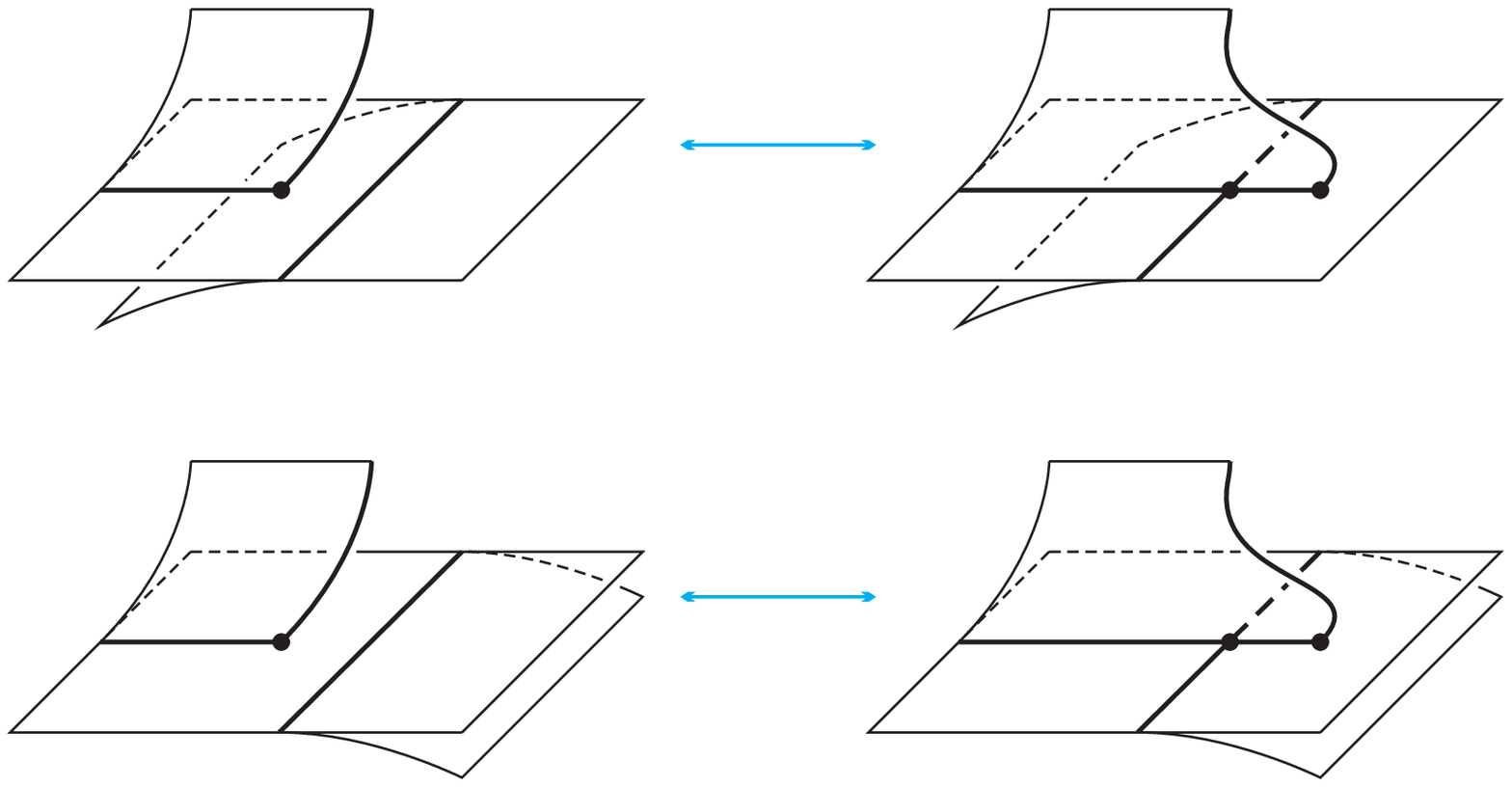}
    \end{center}
\mycap{The spike-sliding moves.\label{slide/spike:fig}}
\end{figure}
%%%%%%%%%%%%%%%%%%%%%%%%%%%%%%%%%%%%%%%%%%%%%%%%%%%%%%%%%

The following result is evident:

\begin{prop}\label{sliding:ok:prop}
If two stream-spines $P_1$ and $P_2$ in $\calS_0$ are related by a sliding move then the corresponding
streams $\varphi_0^*(P_1)$ and $\varphi_0^*(P_2)$ are stream-homotopic to each other.
\end{prop}

\subsection{Translating catastrophes into moves}

In this subsection we establish the following:

\begin{thm}
Let $\varphi_0:\calS_0\to\calF_0$ be the surjection from the set of
stream-spines to the set of streams on $3$-manifolds up to homotopy.
Then $\varphi_0(P_1)$ and $\varphi_0(P_2)$ coincide in $\calF_0$
if and only if $P_1$ and $P_2$ are related by sliding moves.
\end{thm}

\begin{proof}
We must show that the elementary catastrophes along a generic stream-homotopy, as described in
Proposition~\ref{catastrophes:prop}, correspond at the level of stream-spines to the sliding moves.
Checking that the catastrophes of Fig.~\ref{20/catastrophe:fig} and~\ref{32/catastrophe:fig}
correspond to the $0\leftrightarrow2$ and $2\leftrightarrow3$ sliding moves is easy and
already described in~\cite{LNM}, so we do not reproduce the argument.

We then concentrate on the catastrophes of Fig.~\ref{catastrophe/trans/in:fig}, showing that on the associated out-forward spines
their effect is that of a spike-sliding. This is done in Fig.~\ref{slide/trans/out:fig}
\begin{figure}
    \begin{center}
    \includegraphics[scale=.6]{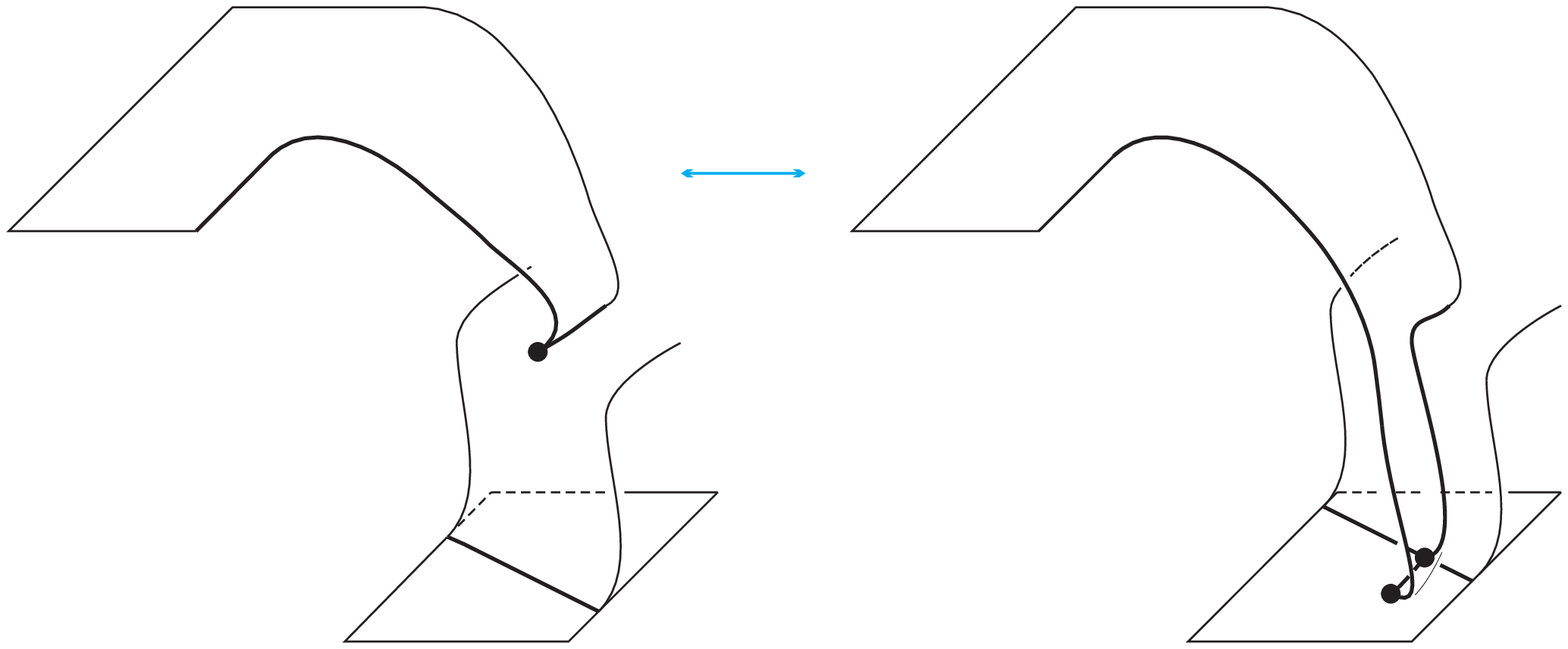}
    \end{center}
\mycap{From a catastrophe involving concave tangency of an incoming transition orbit
to a spike-sliding in the associated out-forward stream-spine.\label{slide/trans/out:fig}}
\end{figure}
for the catastrophe in the top portion of Fig.~\ref{catastrophe/trans/in:fig}, 
which is then easily recognized to give the first spike-sliding
move of Fig.~\ref{slide/spike:fig}; a very similar picture shows that the bottom portion of
Fig.~\ref{catastrophe/trans/in:fig} gives the second spike-sliding move of
Fig.~\ref{slide/spike:fig}.

The proof is now complete and the isomorphism of the in-backward and out-forward stream-spines implies 
that the effect of the catastrophes of Fig.~\ref{catastrophe/trans/in:fig}
is that of a spike-sliding also on the in-backward stream-spine. It is however instructive to
analyze the effect directly on the in-backward stream-spine ---in fact, it is not
even obvious at first sight that the catastrophes of Fig.~\ref{catastrophe/trans/in:fig} have
any impact on the in-backward stream-spine, given that there is no transition orbit to follow backward anyway.
But the catastrophes of Fig.~\ref{catastrophe/trans/in:fig} do have an impact on the
in-backward stream-spine, because at the catastrophe time there is an orbit that from a concave
tangency point traces back to a transition point. To analyze what the impact exactly is,
we restrict to the top portion of Fig.~\ref{catastrophe/trans/in:fig} and we employ
Fig.~\ref{near/transition/in:fig} in a crucial fashion. We do this in
Fig.~\ref{slide/trans/in:fig},
\begin{figure}
    \begin{center}
    \includegraphics[scale=.75]{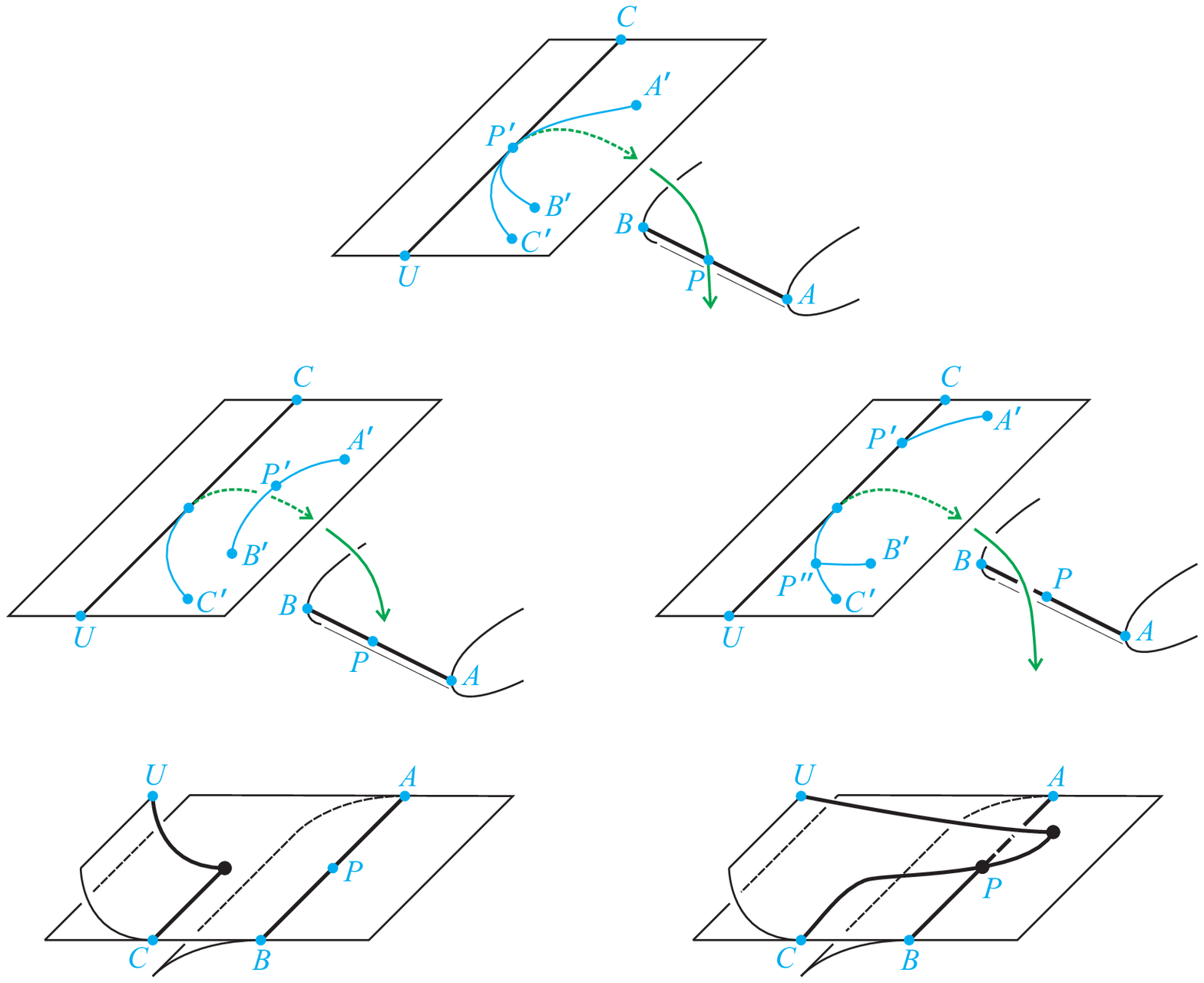}
    \end{center}
\mycap{From a catastrophe involving concave tangency of an incoming transition orbit
to a spike-sliding in the associated in-backward stream-spine.\label{slide/trans/in:fig}}
\end{figure}
where we show the exact time of the catastrophe (top),
the situation before (middle-left) and after (middle-right) the catastrophe, and
the corresponding in-backward stream-spines (bottom). In the middle figures we show how the concave tangency lines trace
back to the in-region, showing for some points $Q$ the boundary point $Q'$ obtained by following backward the orbit through $Q$; note that
after the catastrophe one point $P$ traces back first to a point $P'$ of the concave tangency line and then to a point $P''$ of
the in-region. Using the information of the middle figures one indeed sees that the corresponding
stream-spines are as in the bottom figures, where one recognizes the first spike-sliding
of Fig.~\ref{slide/spike:fig}.
\end{proof}

\section{Combinatorial presentation of generic flows}

As already anticipated, let us now define $\calF$ as the set of pairs $(M,v)$ where $M$ is a
compact, connected, oriented $3$-manifold (possibly without boundary) and $v$ is a generic flow on $M$,
with $M$ viewed up to diffeomorphism and $v$ viewed up to homotopy on $M$ fixed on $\partial M$.
To provide a combinatorial presentation of $\calF$ we call:
\begin{itemize}
\item \emph{Trivial sphere} on the boundary of some $(N,w)$ one that is split into one in-disc and one out-disc
by one concave tangency circle;
\item \emph{Trivial ball} a ball $(B^3,u)$ with $u$ a stream on $B^3$ and $\partial B^3$
split into one in-disc and one out-disc
by one convex tangency circle.
\end{itemize}
Note that a trivial ball can be glued to a trivial sphere matching the vector fields.
We now define $\calS$ as the subset of $\calS_0$ consisting of stream-spines $P$ such that the boundary of
$\varphi_0(P)$ contains at least one trivial sphere. We will establish the following:

\begin{thm}\label{flow:thm}
For $P\in\calS$ let $\varphi(P)$ be obtained from $\varphi_0(P)$ by
attaching a trivial ball to a trivial sphere in the boundary of $\varphi_0(P)$.
This gives a well-defined surjective map $\varphi:\calS\to\calF$, and
$\varphi(P_0)=\varphi(P_1)$ if an only if $P_0$ and $P_1$ are obtained
from each other by the sliding moves of Figg.~\ref{slide/20:fig} to~\ref{slide/spike:fig}.
\end{thm}

\subsection{Equivalence of trivial balls}
In this subsection we will show that the map $\varphi$ of Theorem~\ref{flow:thm} is well-defined.
To this end choose $P\in\calS$ and set $(N,w)=\varphi_0(P)$. To define $\varphi(P)$ we
must choose one trivial sphere $S\subset\partial N$, a trivial ball $(B^3,u)$ and
a diffeomorphism $f:\partial B^3\to S$ matching $u$ to $w$. The manifold $M$
resulting from the gluing is of course independent of $S$, and the resulting
flow $v$ on $M$ is of course independent of $f$ up to homotopy. However, when
the boundary of $\varphi_0(P)$ contains more than one trivial sphere, it
is not obvious that the pair $(M,v)$ as an element of $\calF$ is independent of $S$.
This will be a consequence of the following:

\begin{prop}
Let $v$ be a generic flow on $M$, and let $B_1$ and $B_2$ be disjoint trivial balls contained in
the interior of $M$. Then there is a flow $v'$ on $M$
homotopic to $v$ relatively to $(\partial M)\cup B_1\cup B_2$
such that $\overline{M\setminus B_1}$ and $\overline{M\setminus B_2}$, endowed with the
restrictions of $v'$, are diffeomorphic to each other.
\end{prop}

\begin{proof}
Choose a smooth path $\alpha:[0,1]\to M$ with $\alpha(j)\in\partial B_j$ and
$\dot\alpha(j)=v(\alpha(j))$ not tangent to $\partial B_j$ for $j=0,1$, and
$\alpha(t)\not\in B_1\cup B_2$ for $0<t<1$. Up to small perturbation we can
assume $\dot\alpha(t)\neq -v(\alpha(t))$ for $t\in[0,1]$, and then homotope $v$
on a neighbourhood of $\alpha$ to a flow $v''$ such that
$v''(\alpha(t))=\dot\alpha(t)$ for $t\in[0,1]$. Now we can homotope $v''$ to $v'$
in a neighbourhood of $B_1\cup B_2\cup\alpha$ as suggested in Fig.~\ref{twoballs:fig},
\begin{figure}
    \begin{center}
    \includegraphics[scale=.6]{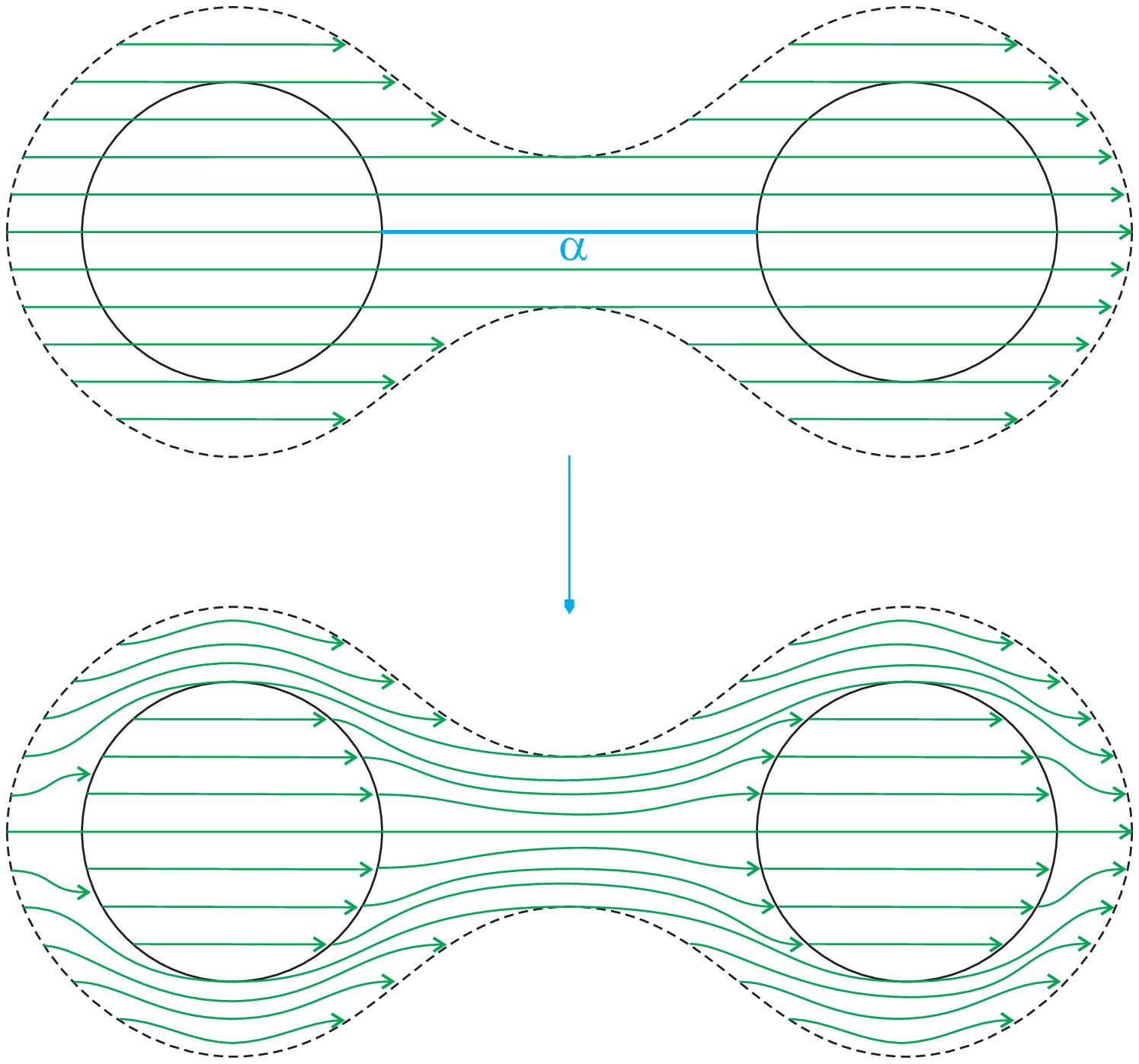}
    \end{center}
\mycap{Homotoping a field so that removing either of two trivial balls gives the same result.\label{twoballs:fig}}
\end{figure}
which gives the desired conclusion.
\end{proof}

\subsection{Normal sections}
Let us now show that the map $\varphi$ of Theorem~\ref{flow:thm} is surjective. To this end we adapt
a definition from~\cite{LNM, Ishii}, calling \emph{normal section} for a manifold $M$ with generic flow $v$
a smooth disc $\Delta$ in the interior of $M$ such that $v$ is transverse to $\Delta$, every orbit
of $v$ meets $\Delta\cup\partial M$ in both positive and negative time, and the orbits of $v$ tangent
to $\partial M$ or intersecting $\partial\Delta$ are generic with respect to each other, with the obvious meaning.
The existence of normal sections is rather easily established~\cite{LNM}, and Fig.~\ref{cutsection:fig}
\begin{figure}
    \begin{center}
    \includegraphics[scale=.6]{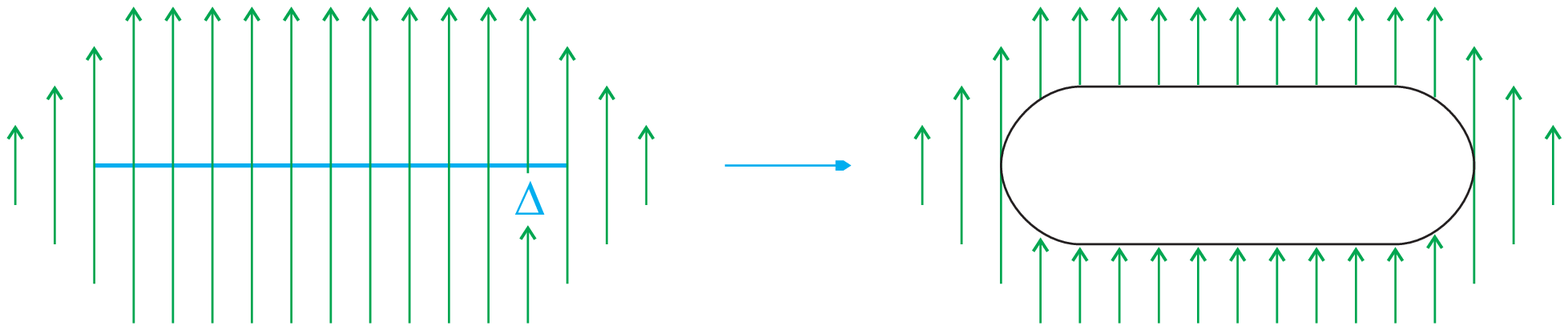}
    \end{center}
\mycap{From a normal section to a stream on the complement of a trivial ball.\label{cutsection:fig}}
\end{figure}
suggests how, given a normal section $\Delta$ of $(M,v)$,
to remove a trivial ball $B$ from $(M,v)$ so that the restriction $w$ of $v$ to $N=M\setminus B$ is a stream on $N$.
By construction if $P$ is a stream-spine such that $\varphi_0^*(P)=(N,w)$ we have that $\varphi(P)$ represents
$(M,v)$, whence the surjectivity of $\varphi$. Let us also note, since we will need this to prove injectivity,
that $P$ can be directly recovered from $(M,v)$ and $\Delta$, taking $\Delta$ union the in-region of $\partial M$ union
the set of points $A$ such that there exists an orbit of $v$ going from $A$ to 
$\partial \Delta$ or to the concave tangency line of $v$ to $\partial M$,
with the obvious branching along triple lines.

\subsection{Homotopy}
We are left to establish injectivity of the
map $\varphi$ of Theorem~\ref{flow:thm}. Recalling that the elements $(M,v)$ of $\calF$ are regarded
up to diffeomorphism of $M$ and homotopy of $v$ on $M$ relative to $\partial M$,
we see that injectivity is a consequence of
the following:

\begin{prop}\label{inj:prop}
Let $(v_t)_{t\in[0,1]}$ be a homotopy of generic flows on $M$, fixed on $\partial M$.
For $j=0,1$ let $\Delta_j$ be a normal section for $(M,v_j)$ and let $P_j$ be the stream-spine
defined by $\Delta_j$ and $v_j$ as at the end of the previous subsection. Then
$P_0$ and $P_1$ are related by the sliding moves of
Figg.~\ref{slide/20:fig} to~\ref{slide/spike:fig}.
\end{prop}

\begin{proof}
The first step is to follow the first normal section
along the homotopy, thus getting a smooth deformation
$(\Sigma_t)_{t\in[0,1]}$ with $\Sigma_0=\Delta_0$ and $\Sigma_t$ a normal section for
$v_t$ for all $t\in[0,1]$. Assuming the deformation is generic, along the deformation
$(\Sigma_t)_{t\in[0,1]}$ and simultaneous homotopy $(v_t)_{t\in[0,1]}$ we will
only have the same catastrophes as in Proposition~\ref{catastrophes:prop},
so $P_0$ and the stream-spine $\widetilde{P_1}$ defined by $\Sigma_1$ and $v_1$
are related by sliding moves.
The next step, as in~\cite{LNM}, consists in constructing normal sections $\Theta$ and $\Xi$ for
$(M,v_1)$ such that $\Sigma_1\cap\Theta=\Theta\cap\Xi=\Xi\cap\Delta_1=\emptyset$, which is easily done.
The conclusion now comes from the fact that given two disjoint normal sections $X$ and $Y$ of
$(M,v)$ one can join them by a small strip constructing a normal section $Z$ that contains $X\cup Y$, and then
one can view the transformation of $X$ into $Y$ as first the smooth expansion of $X$ to $Z$ and then the contraction of $Z$ to $Y$.
At the level of the associated stream-spines this transition again consists of the elementary
sliding moves of Figg.~\ref{slide/20:fig} to~\ref{slide/spike:fig}.
\end{proof}

\begin{rem}\label{glue:rem}
\emph{Suppose for $j=1,2$ that $M_j$ is an oriented $3$-manifold endowed with a generic flow $v_j$, and that $\Sigma_j$ is a boundary
component of $M_j$. Suppose moreover that one is given a diffeomorphism $\Sigma_1\to\Sigma_2$ mapping
the in-region of $\Sigma_1$ to the out-region of $\Sigma_2$ and conversely,
the concave line on $\Sigma_1$ to the convex line on $\Sigma_2$ and conversely,
the concave-to-convex transition points of $\Sigma_1$ to the convex-to-concave transition points of $\Sigma_2$
and conversely.
Then one can glue $M_1$ to $M_2$ along this map, getting on the resulting manifold $M$ a generic flow $v$ well-defined up to homotopy.
This implies that there exists a natural cobordism theory in the set $\calF$ of $3$-manifolds
endowed with a generic flow, and one could hope to use the combinatorial encoding
$\varphi:\calS\to\calF$ described in this paper as a technical tool to develop a TQFT~\cite{Turaev}
for such manifolds.}\end{rem}

\vspace{.5cm}

\noindent
Dipartimento di Matematica\\
Universit\`a di Pisa\\
Via Filippo Buonarroti, 1C\\
56127 PISA -- Italy\\
petronio@dm.unipi.it

\end{document}